\documentclass{scrartcl}
\usepackage[T1]{fontenc}
\usepackage[english]{babel}
\usepackage[latin1]{inputenc}
\usepackage{amsmath}
\usepackage{amsthm,amssymb}
\usepackage{dsfont}
\usepackage{csquotes}
\usepackage{graphicx}
\usepackage{color, pgf}
\usepackage{hyperref}
\usepackage{url}
\usepackage{booktabs}
\usepackage{multirow}
\hypersetup{colorlinks=true, urlcolor=blue, citecolor=red} 
\newtheorem{thm}{Theorem}[section]
\newtheorem{defn}[thm]{Definition}
\newtheorem{lem}[thm]{Lemma}
\newtheorem{cor}[thm]{Corollary}
\newtheorem{rem}[thm]{Remark}
\newtheorem{prop}[thm]{Proposition}

\newcommand{\edgesint}{\mathcal{E}_{\operatorname{int}}}
\newcommand{\reg}{\operatorname{reg}(\mathcal{T})}
\newcommand{\Tau}{\mathcal{T}}
\newcommand{\dkl}{d_{K|L}}

\newcommand{\E}{\mathbb{E}}
\newcommand{\N}{\mathbb{N}}
\newcommand{\R}{\mathbb{R}}
\newcommand{\erww}[1]{\mathbb{E}\left[#1\right]}

\begin{document}

\title{Convergence rates for a finite volume scheme of the stochastic heat equation}
\author{Niklas Sapountzoglou \thanks{Institute of Mathematics, Clausthal University of Technology, 38678 Clausthal-Zellerfeld, Germany \href{mailto:niklas.sapountzoglou@tu-clausthal.de}{\texttt{niklas.sapountzoglou@tu-clausthal.de}}}\and  Aleksandra Zimmermann \thanks{Institute of Mathematics, Clausthal University of Technology, 38678 Clausthal-Zellerfeld, Germany \href{mailto:aleksandra.zimmermann@tu-clausthal.de}{\texttt{aleksandra.zimmermann@tu-clausthal.de}}}}
\date{}

\maketitle

\begin{abstract}
In this contribution, we provide convergence rates for a finite volume scheme of the stochastic heat equation with multiplicative Lipschitz noise and homogeneous Neumann boundary conditions (SHE). More precisely, we give an error estimate for the $L^2$-norm of the space-time discretization of SHE by a semi-implicit Euler scheme with respect to time and a TPFA scheme with respect to space and the variational solution of SHE. The only regularity assumptions additionally needed is spatial regularity of the initial datum and smoothness of the diffusive term.
\end{abstract}
\noindent
\textbf{Keywords:} Finite volume scheme $\bullet$ error estimates 
$\bullet$ stochastic heat equation $\bullet$ \\
multiplicative noise $\bullet$ space-time discretization \\

\noindent
\textbf{Mathematics Subject Classification (2020):} 60H15 $\bullet$ 35K05 $\bullet$ 65M08

\section{Introduction}
Let $\Lambda$ be a bounded, open, connected, and polygonal set of $\mathbb{R}^d$ with $d=2,3$.
Moreover let $(\Omega,\mathcal{A},\mathds{P})$ be a probability space endowed with a right-continuous, complete filtration $(\mathcal{F}_t)_{t\geq 0}$ and let $(W(t))_{t\geq 0}$ be a standard, one-dimensional Brownian motion with respect to $(\mathcal{F}_t)_{t\geq 0}$ on $(\Omega,\mathcal{A},\mathds{P})$. For a measurable space $(\mathcal{S},\mu)$, a separable Banach space $X$ and $1\leq \nu \leq \infty$, we will denote the classical Lebesgue-Bochner space by $L^{\nu}(S;X)$. If $D \subset \R^d$ is open and $k \in \N$ we will write $H^k(D):=\{v \in L^2(D);~\partial_{\alpha} v \in L^2(D)~\forall \alpha \text{ with } |\alpha| \leq k\}$ in the sequel.
For $T>0$, we consider the stochastic heat equation forced by a multiplicative stochastic noise:
\begin{align}\label{equation}
\begin{aligned}
du-\Delta u\,dt &=g(u)\,dW(t), &&\text{ in }\Omega\times(0,T)\times\Lambda;\\
u(0,\cdot)&=u_0, &&\text{ in } \Omega\times\Lambda;\\
\nabla u\cdot \mathbf{n}&=0, &&\text{ on }\Omega\times(0,T)\times\partial\Lambda;
\end{aligned}\tag{SHE}
\end{align}
where $\mathbf{n}$ denotes the unit normal vector to $\partial\Lambda$ outward to $\Lambda$, $u_0\in L^2(\Omega;L^2(\Lambda))$ is $\mathcal{F}_0$-measurable, $\Delta$ denotes the Laplace operator on $H^1(\Lambda)$ associated with the weak formulation of the homogeneous Neumann boundary condition, and $g: \R \to \R$ is a Lipschitz function.

\subsection{State of the art}
In the literature, the existence and uniqueness of variational solutions to Problem \eqref{equation} (see Definition \ref{solution}) is well known and is covered by the classical framework about stochastic parabolic equations, see, \textit{e.g.}, \cite{LR, KryRoz81, PardouxThese, DPraZab}. We are interested in studying numerical schemes for parabolic stochastic partial differential equations (SPDEs) of type \eqref{equation}.\\
In the last decades, numerical schemes for SPDEs have been studied extensively in the literature. For an overview of the state of the art, see, \textit{e.g.}, \cite{ACQS20, DP09, OPW22}. Finite volume schemes have been studied for stochastic scalar conservation laws in, \textit{e.g.}, \cite{BCG16_01, BCG16_02, BCC20}. 
For elliptic and parabolic second-order problems, we recall that finite volume methods allow flexibility on the geometry of the meshes and ensure the local consistency of the numerical fluxes inside the domain, as mentioned in \cite{ABH07}. Furthermore, this kind of discretization is well-suited if one adds a convective first-order term to the equation. If the leading second-order operator in the equation is the Laplacian, one may use standard cell-centered finite volume methods (see, e.g., \cite{EGH00}) such as the two point flux approximation scheme (TPFA). The case of general linear, elliptic, second-order deterministic equations has been studied in \cite{CVV99, DE06, EGH06}. In these cases, the gradient reconstruction procedure requires more than the normal direction between two centers of neighbouring cells and the finite volume scheme becomes more involved, from the theoretical and from the numerical point of view. In this contribution, we restricted ourselves to the case of the Laplacian and the TPFA scheme.
The convergence of a full space-time discretization by a finite volume scheme towards the unique variational solution of the stochastic heat equation in two spatial dimensions has been studied in \cite{BN20} in the case of a linear multiplicative noise and in \cite{BNSZ23} in the more general case of a multiplicative Lipschitz noise. In these contributions, a semi-implicit Euler scheme has been used for the time discretization and a TPFA has been used with respect to space. In \cite{BN20}, the convergence could be proved by classical arguments, since the linear multiplicative noise is compatible with the weak convergences given by the \textit{a-priori} estimates. In \cite{BNSZ23}, the convergence of the nonlinear multiplicative noise term has been addressed using the stochastic compactness method based on the theorems of Prokhorov and Skorokhod. This finite volume scheme has been extended to a stochastic convection-diffusion equation with nonlinearities of zeroth and first order in two or three spatial dimensions in \cite{BNSZ23_2} and \cite{BSZ24}. Here, the convection term is approximated by an upwind scheme and the use of the stochastic compactness method could be avoided, which makes the approach more accessible without deeper knowledge of stochastic analysis.\\
For all these finite volume schemes, convergence has been proved without giving any rate. To the best of our knowledge, there are no results about convergence rates of finite volume schemes for parabolic SPDEs in the variational setting. However, there are many results concerning convergence rates for different numerical schemes of SPDEs, see, \textit{e.g.}, \cite{MajPro18, BP23, KG22} and, in particular, see \cite{GM05, GM07, BHL21, DHW23} for space-time discretizations and convergence rates for nonlinear monotone stochastic evolution equations. In our study, we want to provide error estimates for the finite volume scheme proposed in \cite{BNSZ23} in space dimensions $d=2$ and $d=3$ under rather mild and natural regularity assumptions on the initial condition $u_0$ and on the diffusive term $g$ in the stochastic integral.
\subsection{Aim of the study}
Our aim is to extend the existence and uniqueness result for variational solutions of Problem \eqref{equation} by providing a result about convergence rates of a finite volume scheme for parabolic SPDEs with homogeneous Neumann boundary conditions. According to the existing convergence results, we consider the case of two and three spatial dimensions. We remark that the finite volume scheme proposed in \cite{BNSZ23} does not coincide with the finite element methods and other numerical schemes for which convergence rates for \eqref{equation} have been considered.\\
Our study provides error estimates for the $L^2$-norm of the space-time discretization of \eqref{equation} proposed in \cite{BNSZ23} and the variational solution of \eqref{equation} of order $O(\tau^{1/2} + h + h \tau^{- 1/2})$, where $\tau$ represents the time step and $h$ the spatial parameter. In the deterministic case, see \textit{e.g.} \cite{Flore21}, one may get rid of the unsatisfying term $h \tau^{-1/2}$ by assuming the exact solution to be two times differentiable in time. Since we cannot assume more than continuity in time in the stochastic case, the stochastic nature of Problem \eqref{equation} creates worse convergence rates compared to the deterministic case. Convergence rates of order $O(h\tau^{-1/2})$ do also appear for mixed finite element schemes for the stochastic Navier-Stokes equation (see \cite{FQ21}).\\
The main idea of the work is to compare the exact solution of Problem \eqref{equation} with the solution of the semi-implicit Euler scheme first. In order to do this, we need to assume $H^2$-regularity in space for the initial value $u_0$ and smoothness for the function $g$ in the stochastic It\^{o} integral. 
This error turns out to be of order $O(\tau^{1/2})$. Then, we compare the solution of the semi-implicit Euler scheme with its parabolic projection (see Definition \ref{Definition parabolic projection}), whose error is of order $O(h)$. Finally, comparing the parabolic projection of the solution of the semi-implicit Euler scheme with the solution of the finite volume scheme provides an error of order $O(h + h \tau^{-1/2})$.
\subsection{Outline}
The paper is organized as follows. In Section 2 we introduce the considered problem as well as the finite volume framework. Moreover, we provide some analytic tools which are used frequently throughout the paper.\\
In Section 3 we state the main result and the regularity assumptions.\\
In Section 4 we introduce the semi-implicit Euler scheme for the stochastic heat equation with Neumann boundary conditions and we provide estimates for its solution.\\
Section 5 contains stability estimates for the exact solution of the stochastic heat equation and a regularity result.\\
In Section 6 we prove error estimates, more precisely we estimate the $L^2$-error between the exact solution and the solution of the semi-implicit Euler scheme, the $L^2$-error between the solution of the semi-implicit Euler scheme and its parabolic projection and finally the $L^2$-error between the parabolic projection and the solution of the finite volume scheme.\\
In Section 7, we collect the previous results and prove our main theorem.\\
Finally, we present our computational experiments in Section 8.
\section{A Finite Volume scheme for the stochastic heat equation}
\subsection{The stochastic heat equation with multiplicative noise}
For the well-posedness of a variational solution to the stochastic heat equation, we assume the following regularity on the data:
\begin{itemize}
\item[$i.)$] $u_0\in L^2(\Omega;L^2(\Lambda))$ is $\mathcal{F}_0$-measurable.
\item[$ii.)$] $g:\mathbb{R}\rightarrow\mathbb{R}$ is a Lipschitz continuous function and, in particular, there exists a constant $C_g>0$ such that $|g(r)|^2 \leq C_g(1+r^2)$ for all $r\in\mathbb{R}$.\\
\end{itemize}
For $u_0$ and $g$ that satisfy the above assumptions, we will be interested in the following concept of \textit{variational solution} for \eqref{equation}:
\begin{defn}\label{solution} A predictable stochastic process $u$ in $L^2\left(\Omega\times(0,T);L^2(\Lambda)\right)$ is 
a variational solution to Problem \eqref{equation} if it belongs to 
\begin{align*}
 L^2(\Omega;\mathcal{C}([0,T];L^2(\Lambda)))\cap L^2(\Omega;L^2(0,T;H^1(\Lambda)))
\end{align*}
and satisfies, for all $t\in[0,T]$,
\begin{align}\label{240311_eq1}
u(t)-u_0-\int_0^t \Delta u(s)\,ds=\int_0^t g(u(s))\,dW(s)
\end{align}
in $L^2(\Lambda)$ and $\mathds{P}$-a.s. in $\Omega$, where $\Delta$ denotes the Laplace operator on $H^1(\Lambda)$ associated with the weak formulation of the homogeneous Neumann boundary condition.
\end{defn}
\begin{rem}
Existence, uniqueness and regularity of a variational solution to \eqref{equation} in the sense of Definition \ref{solution} is well-known in the literature, see, e.g., \cite{PardouxThese, KryRoz81, LR}.
\end{rem}
\subsection{Admissible meshes and notations}
Let us introduce the temporal and spatial discretizations. For the time-discretization, let $N\in\mathbb{N}$ be given. We define the equidistant time step $\tau=\frac{T}{N}$ and divide the interval $[0,T]$ in $0=t_0<t_1<...<t_N=T$ with $t_n=n \tau$ for all $n\in \{0, ..., N-1\}$.
For the space discretization, we consider admissible meshes in the sense of \cite{EGH00}, Definition 9.1. For the sake of self-containedness, we add the definition:
\begin{defn}[see \cite{EGH00}, Definition 9.1]\label{defmesh} 
For a bounded domain $\Lambda\subset\mathbb{R}^d$, $d=2$ or $d=3$, an admissible finite-volume mesh $\mathcal{T}$ is given by a family of open, polygonal, and convex subsets $K$, called \textit{control volumes}
and a family of subsets of $\overline{\Lambda}$, contained in hyperplanes of $\mathbb{R}^d$ with strictly positive $(d-1)$-dimensional Lebesgue measure, denoted by $\mathcal{E}$. The elements of $\mathcal{E}$ are called \textit{edges} (for $d=2$) or \textit{sides} (for $d=3$) of the control volumes. Finally we associate a family $\mathcal{P}=(x_K)_{K\in\mathcal{T}}$ of points in $\Lambda$ to the family of control volumes, called \textit{centers}. 
We assume $K\in\mathcal{T}$, $\mathcal{E}$ and $\mathcal{P}$ to satisfy the following properties:
\begin{itemize}
\item $\overline{\Lambda}=\bigcup_{K\in\mathcal{T}}\overline{K}$.
\item For any $K\in\mathcal{T}$, there exists a subset $\mathcal{E}_K$ of $\mathcal{E}$ such that $\partial K=\overline{K}\setminus K=\bigcup_{\sigma\in \mathcal{E}_K}\overline{\sigma}$ and $\mathcal{E}=\bigcup_{K\in\mathcal{T}}\mathcal{E}_K$. $\mathcal{E}_K$ is called the set of edges of $K$ for $d=2$ and sides for $d=3$, respectively.
\item If $K,L\in\mathcal{T}$ with $K\neq L$ then either the $(d-1)$-dimensional Lebesgue measure of $\overline{K}\cap\overline{L}$ is $0$  or $\overline{K}\cap\overline{L}=\overline{\sigma}$ for some $\sigma\in \mathcal{E}$, which will then be denoted by $K|L$.
\item $\mathcal{P}=(x_K)_{K\in\mathcal{T}}$ is such that $x_K\in \overline{K}$ for all $K\in\mathcal{T}$ and, if $K,L\in\mathcal{T}$ are two neighboring control volumes, $x_K\neq x_L$ and the straight line between the centers $x_K$ and $x_L$ is orthogonal to the edge $\sigma=K|L$.
\item For any $\sigma\in\mathcal{E}$ such that $\sigma\subset\partial\Lambda$, let $K$ be the control volume such that $\sigma \in\mathcal{E}_K$. If $x_K\notin\sigma$, the straight line going through $x_K$ and orthogonal to $\sigma$ has a nonempty intersection with $\sigma$.
\end{itemize}
\end{defn}
Examples of admissible meshes are triangular meshes for $d=2$, where the condition that angles of the triangles are less than $\frac{\pi}{2}$ ensures $x_K\in K$. For $d=3$, Voronoi meshes are admissible. See \cite{EGH00}, Example 9.1 and 9.2 for more details.
In the following, we will use the following notation:
\begin{itemize}
\item $h=\operatorname{size}(\mathcal{T})=\sup\{\operatorname{diam}(K): K\in\mathcal{T}\}$ the mesh size.
\item $\mathcal{E}_{\operatorname{int}}:=\{\sigma\in\mathcal{E}:\sigma\nsubseteq \partial\Lambda\}$, $\mathcal{E}_{\operatorname{ext}}:=\{\sigma\in\mathcal{E}:\sigma\subseteq \partial\Lambda\}$, $\mathcal{E}^{K}_{\operatorname{int}}=\mathcal{E}_{\operatorname{int}}\cap \mathcal{E}_{K}$ for $K\in\mathcal{T}$.

\item For $K\in\mathcal{T}$, $\sigma\in\mathcal{E}$ and $d=2$ or $d=3$, let $m_K$ be the $d$-dimensional Lebesgue measure of $K$ and let $m_\sigma$ be the $(d-1)$ dimensional Lebesgue measure of $\sigma$.

\item For $K\in\mathcal{T}$, $\mathbf{n}_K$ denotes the unit normal vector to $\partial K$ outward to $K$ and for $\sigma \in \mathcal{E}_K$, we denote the unit vector on the edge $\sigma$ pointing out of $K$ by $\mathbf{n}_{K,\sigma}$.

\item Let $K,L\in\mathcal{T}$ be two neighboring control volumes. For $\sigma=K|L\in\edgesint$, let $d_{K|L}$ denote the Euclidean distance between $x_K$ and $x_L$.
\end{itemize}

Using these notations, we introduce a positive number 
\begin{align}\label{mrp}
\reg=\max\left(\mathcal N,\max_{\scriptscriptstyle K \in\mathcal{T} \atop \sigma\in\mathcal{E}_K} \frac{\operatorname{diam}(K)}{d(x_K,\sigma)}\right)
\end{align}
(where $\mathcal N$ is the maximum of edges incident to any vertex) that measures the regularity of a given mesh and is useful to perform the convergence analysis of finite-volume schemes.
This number should be uniformly bounded by a constant $\chi>0$ not depending on the mesh size $h$ for the convergence results to hold, \textit{i.e.},
\begin{align}\label{mesh_regularity}
\reg \leq \chi.
\end{align}
We have in particular $\forall K,L\in \mathcal{T}$,  
\begin{align*}
\frac{h}{\dkl}\leq \reg.
\end{align*} 

\subsection{The finite volume scheme}

Let $\mathcal{T}_h$ be an admissible finite volume mesh with mesh size $h>0$. For a $\mathcal{F}_0$-measurable random element $u_0\in L^2(\Omega;L^2(\Lambda))$ we define $\mathbb{P}$-a.s. in $\Omega$ its piecewise constant spatial discretization $u_h^0= \sum_{K \in \mathcal{T}_h} u_K^0 \mathds{1}_K$, where
\begin{align*}
u_K^0= \frac{1}{m_K} \int_K u_0(x) \, dx, ~~\forall K \in \mathcal{T}_h.
\end{align*}
The finite-volume scheme we propose reads, for the initially given $\mathcal{F}_0$-measurable random vector $(u_K^0)_{K\in\mathcal{T}_h}$ induced by $u_h^0$, as follows:
For any $n \in \{1,\cdots,N\}$, given the $\mathcal{F}_{t_{n-1}}$-measurable random vector $(u^{n-1}_K)_{K\in\mathcal{T}_h}$ we search for a $\mathcal{F}_{t_{n}}$-measurable random vector $(u^{n}_K)_{K\in\mathcal{T}_h}$ such that, for almost every $\omega\in\Omega$, $(u^{n}_K)_{K\in\mathcal{T}_h}$ is solution to the following random equations
\begin{align}\label{FVS}
m_K(u_K^n - u_K^{n-1}) + \tau \sum\limits_{\sigma \in \mathcal{E}_K^{int}} \frac{m_{\sigma}}{d_{K|L}} (u_K^n - u_L^n)= m_K g(u_K^{n-1}) \Delta_nW, ~~~\forall K \in \mathcal{T}_h,
\end{align}
where $0=t_0<t_1<...<t_N=T$, $\tau=t_n - t_{n-1}=T/N$, and $\Delta W_n:= W(t_n)-W(t_{n-1})$ for all $n \in \{1,...,N\}$.\\
\begin{rem}
	The second term on the left-hand side of \eqref{FVS} is the classical two-point flux approximation of the Laplace operator, (see~\cite[Section 10]{EGH00} for more details on the two-point flux approximation of the Laplace operator with Neumann boundary conditions). 
\end{rem}
With any finite sequence of unknowns $(w_{K})_{K\in\Tau_h}$, we may associate the piecewise constant function
\[w_h(x)=\sum_{K\in \Tau_h} w_K\mathds{1}_K(x), ~x\in\Lambda,\]
where $\mathds{1}_K$ denotes the standard indicator function on $K$, i.e., $\mathds{1}_K(x)=1$ for $x \in K$ and $\mathds{1}_K(x)=0$ for $x \notin K$.
The discrete $L^2$-norm for $w_h$ is then given by
\[\Vert w_h\Vert_2^2=\sum_{K\in\Tau_h}m_K |w_K|^2\]
and we may define its discrete $H^1$-seminorm by
\[
|w_h|_{1,h}^2:= \sum_{\sigma\in\mathcal{E}_{\operatorname{int}}} \frac{m_{\sigma}}{d_{K|L}} (w_K-w_L)^2.\]
Consequently, if $(u^{n}_K)_{K\in\mathcal{T}_h}$ is the solution to \eqref{FVS}, we will consider 
\[u_h^n= \sum_{K \in \mathcal{T}_h} u_K^n \mathds{1}_K\]
for any $n \in \{1,...,N\}$ a.s. in $\Omega$.
\subsection{Tools from discrete analysis}
In the sequel, the following discrete Gronwall inequality will be frequently used:

\begin{lem}[\cite{S69}, Lemma 1]\label{Gronwall}
Let $N \in \mathbb{N}$ and $a_n, b_n, \alpha \geq 0$ for all $n \in \{1,...,N\}$. Assume that for every $n \in \{1,...,N\}$
\begin{align*}
a_n \leq \alpha + \sum\limits_{k=0}^{n-1} a_k b_k.
\end{align*}
Then, for any $n \in \{1,...,N\}$ we have
\begin{align*}
a_n \leq \alpha \exp\bigg(\sum\limits_{k=0}^{n-1}b_k \bigg).
\end{align*}
\end{lem}

Moreover, we use the following version of discrete Poincar\'{e} inequality:
\begin{lem}[\cite{B_CC_HF15}, Theorem 3.6, see also \cite{Flore21}, Lemma 1]\label{Lemma PI}
There exists a constant $C_p>0$, only depending on $\Lambda$, such that for all admissible meshes $\mathcal{T}_h$ and for all piecewise constant functions $w_h=(w_K)_{K \in \mathcal{T}_h}$ we have
\begin{align}\label{291123_01}
\Vert w_h \Vert_2^2 \leq C_p |w_h|_{{1, h}}^2 + 2|\Lambda|^{-1} \bigg( \int_{\Lambda} w_h(x) \, dx \bigg)^2. \tag{PI}
\end{align}
\end{lem}
Furthermore, we will use the upcoming version of the discrete integration by parts rule
\begin{rem}[\cite{BNSZ23}, Remark 2.8]\label{DIBP}
Consider two finite sequences $(w_K)_{K \in \Tau_h}, (v_K)_{K \in \Tau_h}$. Then
\begin{align*}
\sum\limits_{K \in \Tau_h} \sum\limits_{\sigma\in\mathcal{E}_K^{\operatorname{int}}} \frac{m_{\sigma}}{d_{K|L}} (w_K - w_L)v_K = \sum\limits_{\sigma\in\mathcal{E}_{\operatorname{int}}} \frac{m_{\sigma}}{d_{K|L}} (w_K - w_L)(v_K - v_L).
\end{align*}
\end{rem}
\section{Regularity assumptions on the data and main result}
In order to prove our result, we need the following additional regularity assumptions on the data:
\begin{itemize}
    \item[$(R1)$] $~g \in \mathcal{C}^2(\mathbb{R})$ such that $g'$ and $g''$ are bounded on $\mathbb{R}$,
    \item[$(R2)$] $~u_0 \in L^2(\Omega;H^2(\Lambda))$ is $\mathcal{F}_0$-measurable and satisfies the weak homogeneous Neumann boundary condition.
\end{itemize}

\begin{thm}\label{main theorem}
Let $(R1)$ and $(R2)$ be satisfied and let $u$ be the variational solution of the stochastic heat equation \eqref{equation}
in the sense of Definition \ref{solution}. 
For $N \in \mathbb{N}$ and $h>0$, let $\Tau_h$ be an admissible mesh and $(u_h^n)_{n=1,...,N}$ the solution of the finite volume scheme \eqref{FVS}. Then, there exists a constant $\Upsilon>0$ depending on the mesh regularity $\operatorname{reg}(\Tau_h)$ but not depending on $n, N$ and $h$ explicitly such that
\begin{align*}
\sup_{t \in [0,T]} \mathbb{E} \Vert u(t) - u_{h,N}^r(t) \Vert_2^2 \leq \Upsilon(\tau + h^2 + \frac{h^2}{\tau}),
\end{align*}
where $u_{h,N}^r(t)=\sum\limits_{n=1}^N u_h^n \mathds{1}_{[t_{n-1}, t_n)}(t)$, $u_{h,N}^r(T)= u_h^N$. If \eqref{mesh_regularity} is satisfied, $\Upsilon$ may depend on $\chi$ and the dependence of $\Upsilon$ on $\operatorname{reg}(\Tau_h)$ can be omitted.
\end{thm}

\begin{rem}
The error estimate of Theorem \ref{main theorem} also applies to  
\begin{align*}
u_{h,N}^l(t)=\sum\limits_{n=1}^N u_h^{n-1} \mathds{1}_{[t_{n-1}, t_n)}(t), ~u_{h,N}(T)=u_h^N.
\end{align*}
It is only necessary to check that for any $t \in [0,t_1)$ we have
\begin{align*}
\mathbb{E} \Vert u(t) - u_h^0 \Vert_2^2 \leq \Upsilon(\tau + h^2 + \frac{h^2}{\tau}).
\end{align*}
Eventually, we obtain by applying Lemma \ref{Lemma 220923_01} and Lemma \ref{Lemma 240124_01}:
\begin{align*}
&\mathbb{E} \Vert u(t) - u_h^0 \Vert_2^2 \\
&\leq 2 \mathbb{E}\Vert u(t) -u_0 \Vert_2^2 + 2 \mathbb{E}\Vert u_0 - u_h^0 \Vert_2^2  \\
&\leq 2K_4 \tau +  2 \E \sum\limits_{K \in \mathcal{T}} \int_K |u_0(x) - \frac{1}{m_K} \int_K u_0(y) \, dy |^2 \, dx  \\
&=  2K_4 \tau +  2 \E \sum\limits_{K \in \mathcal{T}} \int_K \frac{1}{{m_K}^2} | \int_K u_0(x) -  u_0(y) \, dy |^2 \, dx \\
&\leq  2K_4 \tau+  2 \E \sum\limits_{K \in \mathcal{T}} \frac{1}{m_K} \int_K C(2, \Lambda) h^2 \Vert u_0 \Vert_{H^2(K)}^2 \, dx \\
&\leq 2K_4 \tau+ 2 C(2, \Lambda) \E \sum\limits_{K \in \mathcal{T}} h^2 \Vert u_0 \Vert_{H^2(K)}^2   \\
&= 2K_4 \tau+ 2 C(2, \Lambda) \E \Vert u_0 \Vert_{H^2(\Lambda)}^2 h^2.
\end{align*}
\end{rem}
\begin{rem}
		As already mentioned in \cite{MajPro18}, it is easily possible to generalize the error analysis of our scheme to a cylindrical Wiener process in $L^2(D)$. In this case, we consider a diffusion operator $G(u)$, where $G$ is a Hilbert-Schmidt operator on $L^2(D)$ such that, if $(e_k)_k$ is an orthonormal basis of $L^2(D)$, for any $k\in\mathbb{N}$, $G(u)(e_k)=g_k(u)$, where $g_k\in\mathcal{C}^2(\mathbb{R})$ satisfies
		\[\sum_{k=1}^{\infty}\Vert g'_k\Vert^2_{\infty}<+\infty, \quad  \sum_{k=1}^{\infty}\Vert g''_k\Vert^2_{\infty}<+\infty.\]
		Since the noise is discretized with respect to the time variable, we limit ourselves to the $1$-d case so as not to overload the notation.
\end{rem}
\section{Semi-implicit Euler scheme for the stochastic heat equation}
Now, for a $\mathcal{F}_0$-measurable random variable $v^0\in L^2(\Omega; L^2(\Lambda))$ we consider the semi-implicit Euler scheme
\begin{align}\label{ES}
\begin{cases} v^n - v^{n-1} - \tau \Delta v^n = g(v^{n-1}) \Delta_n W ~&\text{in}~ \Omega \times \Lambda ,\\
\nabla v^n \cdot \mathbf{n} =0~&\text{on}~\Omega \times \partial \Lambda,
\end{cases} \tag{ES}
\end{align}
where $0=t_0<t_1<...<t_N=T$ and $\tau=t_n - t_{n-1}=T/N$ for all $n \in \{1,...,N\}$.\\
From the Theorem of Stampaccia (see, e.g, \cite{BrezisFA}, Thm. 5.6) or arguing as in \cite{SWZ19} it follows that there exists a unique $(\mathcal{F}_{t_n})$-measurable solution $v^n$ in $L^2(\Omega;H^1(\Lambda))$ to \eqref{ES} in the weak sense.\\ 
Moreover, from Theorem 3.2.1.3 in \cite{Grisvard85} we obtain $v^n \in L^2(\Omega; H^2(\Lambda))$, hence \eqref{ES} holds a.e. in $\Omega\times \Lambda$ and the homogeneous Neumann boundary condition holds in the weak sense, i.e., for any $w \in H^1(\Lambda)$ we have
\begin{align*}
\int_{\partial \Lambda} w \nabla v^n \cdot \mathbf{n} \, dS=0,
\end{align*}
where the integrands $\nabla v^n$ and $w$ have to be understood in the sense of trace operators.\\
Furthermore, in the case $d=2$, Lemma 4.3.1.2, Lemma 4.3.1.3 and Theorem 4.3.1.4 in \cite{Grisvard85} yield the existence of a constant $C>0$ only depending on $\Lambda$ such that for any random variable $u : \Omega \to H^2(\Lambda)$ satisfying the weak homogeneous Neumann boundary conditions we have
\begin{align}\label{eq 041023_02}
\Vert u \Vert_{H^2(\Lambda)}^2 \leq C(\Vert \Delta u \Vert_2^2 + \Vert u \Vert_2^2) ~~~\mathbb{P}\text{-a.s. in}~\Omega.
\end{align}
Actually, inequality \eqref{eq 041023_02} still holds true for any $d \in \mathbb{N}$ and it can be chosen $C= 12$, see Appendix, Lemma \ref{Theorem 110124_01}.
\begin{lem}\label{Lemma 250923_01}
There exists a constant $K_1>0$ such that for any $N \in \mathbb{N}$ we have
\begin{align*}
\sup_{n \in \{1,...,N\}} \mathbb{E} \Vert v^n \Vert_2^2 + \sum\limits_{n=1}^N \mathbb{E} \Vert v^n - v^{n-1} \Vert_2^2\leq K_1.
\end{align*}
\end{lem}
\begin{proof}
Using $v^n$ as a test function in the first equation of \eqref{ES} yields
\begin{align}\label{eq 250923_01}
I + II = III,
\end{align}
where
\begin{align*}
I&= \frac{1}{2} \bigg(\Vert v^n \Vert_2^2 - \Vert v^{n-1} \Vert_2^2 + \Vert v^n - v^{n-1} \Vert_2^2 \bigg) = (v^n - v^{n-1}, v^n)_2, \\
II&= \tau \Vert \nabla v^n \Vert_2^2, \\
III&= (g(v^{n-1}) \Delta_n W, v^n)_2.
\end{align*}
Taking the expectation in $III$, using the $\mathcal{F}_{t_{n-1}}$-measurability of $v^{n-1}$, applying Young's inequality and It\^{o} isometry (see, e.g., \cite{DPraZab}, p.101, equation (4.30)) yields
\begin{align*}
&|\mathbb{E} (III) |= |\mathbb{E} (g(v^{n-1} \Delta_n W, v^n-v^{n-1})_2| \leq \tau \mathbb{E} \Vert g(v^{n-1}) \Vert_2^2 + \frac{1}{4} \mathbb{E} \Vert v^n - v^{n-1} \Vert_2^2 \\
&\leq \tau C_g \bigg( |\Lambda| + \mathbb{E} \Vert v^{n-1} \Vert_2^2\bigg) + \frac{1}{4} \mathbb{E} \Vert v^n - v^{n-1} \Vert_2^2.
\end{align*}
Now, taking the expectation in \eqref{eq 250923_01} and taking the sum over $n=1,...,m$ for any $m \in \{1,...,N\}$ we obtain
\begin{align}\label{eq 041023_01}
&\frac{1}{2} \mathbb{E} \Vert v^m \Vert_2^2+ \frac{1}{4} \sum\limits_{n=1}^m \mathbb{E} \Vert v^n - v^{n-1} \Vert_2^2 + \tau \sum\limits_{n=1}^m \mathbb{E} \Vert \nabla v^n \Vert_2^2 \\
&\leq \frac{1}{2} \mathbb{E} \Vert v^0 \Vert_2^2 + T C_g |\Lambda| + \tau C_g \sum\limits_{n=1}^m \mathbb{E} \Vert v^{n-1} \Vert_2^2. \notag
\end{align}
Then, Lemma \ref{Gronwall} yields the existence of $\tilde{K}_1>0$ such that
\begin{align*}
\sup_{n \in \{1,...,N\}} \mathbb{E} \Vert v^n \Vert_2^2 \leq \tilde{K}_1.
\end{align*}
Now, using \eqref{eq 041023_01} with $m=N$ yields the assertion.
\end{proof}
\begin{lem}\label{Lemma 250923_02}
For $v_0 \in L^2(\Omega; H^1(\Lambda))$ $\mathcal{F}_0$-measurable and $r \in \N_0$, there exists a constant $\tilde{K}_2= \tilde{K}_2(r)>0$ such that for any $N \in \mathbb{N}$ we have
\begin{align}\label{eq 240312_1}
\begin{aligned}
\sup_{n \in \{1,...,N\}} &\erww{\Vert v^n \Vert_{H^1}^{2^{r+1}}}\\ 
&+ \sum\limits_{n=1}^N \mathbb{E}\left[\prod\limits_{l=1}^r \Big[\Vert v^n \Vert_{H^1}^{2^l} + \Vert v^{n-1} \Vert_{H^1}^{2^l} \Big] \times \Big(\Vert v^n - v^{n-1} \Vert_{H^1}^2 + \tau \Vert \Delta v^n \Vert_2^2 \Big)\right] \leq \tilde{K}_2.
\end{aligned}
\end{align}
Especially, there exists a constant $K_2>0$ such that
\begin{align*}
\sup_{n \in \{1,...,N\}} \erww{\Vert v^n \Vert_2^4 }+ \sup_{n \in \{1,...,N\}} \erww{ \Vert \nabla v^n \Vert_2^{10}} + \tau \sum\limits_{n=1}^N \erww{ \Vert \Delta v^n \Vert_2^2 \Vert \nabla v^n \Vert_2^2} \leq K_2.
\end{align*}

\end{lem}
\begin{proof}
The proof is similar to the proof of Lemma 4.1 in \cite{MajPro18}, however, for the sake of completeness, we want to give the whole proof by induction.\\
We firstly claim the following: For any $r \in \N_0$ and any $n=1,\ldots,N$ there exists $K(r)>0$ and a $(\mathcal{F}_{t_{n-1}})$-measurable real-valued random variable $f_{n-1}^r$ with $|f_{n-1}^r|^2 \leq \tilde{C}(r)(1+ \Vert v^{n-1} \Vert_{H^1}^{2^{r+2}})$ for some constant $\tilde{C}(r)$ only depending on $r$ such that
\begin{align}\label{eq 020224_01}
 & \Big[\Vert v^n \Vert_{H^1}^{2^{r+1}} - \Vert v^{n-1} \Vert_{H^1}^{2^{r+1}}\Big] +  \bigg[\prod\limits_{l=1}^r \Big[\Vert v^n \Vert_{H^1}^{2^l} + \frac{1}{2} \Vert v^{n-1} \Vert_{H^1}^{2^l} \Big] \times \Big( \frac{1}{2} \Vert v^n - v^{n-1} \Vert_{H^1}^2 + \tau \Vert \Delta v^n \Vert_2^2 \Big)\bigg] \notag \\
&\leq C(r) \sum\limits_{l=1}^{r+1}|\Delta_n W|^{2^l} (1+ \Vert v^{n-1} \Vert_{H^1}^{2^{r+1}}) +f_{n-1}^r \Delta_n W
\end{align}
$\mathbb{P}$-a.s. in $\Omega$. We prove \eqref{eq 020224_01} for any $r \in \N_0$ and any $n=1,\ldots,N$ by induction over $r$. For $r=0$ we have to prove
\begin{align*}
 & \Vert v^n \Vert_{H^1}^2 - \Vert v^{n-1} \Vert_{H^1}^2 +  \frac{1}{2} \Vert v^n - v^{n-1} \Vert_{H^1}^2 + \tau \Vert \Delta v^n \Vert_2^2 \\
&\leq C(0) |\Delta_n W|^2 (1+ \Vert v^{n-1} \Vert_{H^1}^2) +f_{n-1}^0 \Delta_n W
\end{align*}
$\mathbb{P}$-a.s. in $\Omega$, where we use the convention $\prod\limits_{l=1}^0=1$. Applying $v^n - \Delta v^n$ to \eqref{ES} yields
\begin{align*}
\frac{1}{2} &\Big[\Vert v^n \Vert_{H^1}^2 - \Vert v^{n-1} \Vert_{H^1}^2+ \Vert v^n - v^{n-1} \Vert_{H^1}^2 \Big] + \tau \Big( \Vert \nabla v^n \Vert_2^2 + \Vert \Delta v^n \Vert_2^2 \Big) \\
&=(g(v^{n-1}) \Delta_n W,  v^n)_2 + (\nabla g(v^{n-1}) \Delta_n W, \nabla v^n)_2 \\
&=(g(v^{n-1}) \Delta_n W,  v^n- v^{n-1})_2 + (\nabla g(v^{n-1}) \Delta_n W, \nabla v^n - \nabla v^{n-1})_2 \\
&+(g(v^{n-1}) \Delta_n W,  v^{n-1})_2 + (\nabla g(v^{n-1}) \Delta_n W, \nabla v^{n-1})_2 \\
&\leq \Vert g(v^{n-1}) \Vert_2^2 |\Delta_n W|^2 + \frac{1}{4} \Vert v^n - v^{n-1} \Vert_2^2 + \Vert \nabla g(v^{n-1}) \Vert_2^2 |\Delta_n W|^2 + \frac{1}{4} \Vert \nabla v^n - \nabla v^{n-1} \Vert_2^2 \\
&+ \bigg( \int_{\Lambda} g(v^{n-1}) v^{n-1} + g'(v^{n-1}) |\nabla v^{n-1}|^2 \, dx \bigg) \Delta_n W \\
&\leq \frac{1}{2}C(0) |\Delta_n W|^2 (1+ \Vert v^{n-1} \Vert_{H^1}^2) + \frac{1}{4} \Vert v^n - v^{n-1} \Vert_{H^1}^2 + \frac{1}{2} f_{n-1}^0 \Delta_n W
\end{align*}
for some constant $C(0)>0$ only depending on $C_g,\Vert g' \Vert_{\infty}$, $\Lambda$ and $f_{n-1}^0 := 2\big( \int_{\Lambda} g(v^{n-1}) v^{n-1} + g'(v^{n-1}) |\nabla v^{n-1}|^2 \, dx \big)$. Subtracting $\frac{1}{4} \Vert v^n - v^{n-1} \Vert_{H^1}^2$, multiplying by $2$ and discarding nonnegative terms on the left-hand side yields \eqref{eq 020224_01} for $r=0$.\\
Now, we assume \eqref{eq 020224_01} holds true for some $r \in \N_0$ and we show that \eqref{eq 020224_01} holds true for $r+1$. Multiplying \eqref{eq 020224_01} with $(\Vert v^n \Vert_{H^1}^{2^{r+1}} + \frac{1}{2} \Vert v^{n-1} \Vert_{H^1}^{2^{r+1}} )$ yields
\begin{align*}
&\frac{3}{4} \bigg(\Vert  v^n \Vert_{H^1}^{2^{r+2}} - \Vert v^{n-1} \Vert_{H^1}^{2^{r+2}}\bigg) + \frac{1}{4} \bigg(\Vert  v^n \Vert_{H^1}^{2^{r+1}} - \Vert  v^{n-1} \Vert_{H^1}^{2^{r+1}}\bigg)^2 \\
&+  \bigg[\prod\limits_{l=1}^{r+1} \Big[\Vert v^n \Vert_{H^1}^{2^l} + \frac{1}{2} \Vert v^{n-1} \Vert_{H^1}^{2^l} \Big] \times \Big( \frac{1}{2} \Vert v^n - v^{n-1} \Vert_{H^1}^2 + \tau \Vert \Delta v^n \Vert_2^2 \Big)\bigg] \notag \\
&\leq \bigg(C(r) \sum\limits_{l=1}^{r+1}|\Delta_n W|^{2^l} (1+ \Vert v^{n-1} \Vert_{H^1}^{2^{r+1}})\bigg) (\Vert v^n \Vert_{H^1}^{2^{r+1}} + \frac{1}{2} \Vert v^{n-1} \Vert_{H^1}^{2^{r+1}} ) \\
&+\bigg( f_{n-1}^r \Delta_n W\bigg) (\Vert v^n \Vert_{H^1}^{2^{r+1}} + \frac{1}{2} \Vert v^{n-1} \Vert_{H^1}^{2^{r+1}} ) \\
&\leq 2C(r)^2 \bigg(\sum\limits_{l=1}^{r+1}|\Delta_n W|^{2^l}\bigg)^2 (1+ \Vert v^{n-1} \Vert_{H^1}^{2^{r+1}})^2 + \frac{1}{8} \bigg(\Vert  v^n \Vert_{H^1}^{2^{r+1}} - \Vert  v^{n-1} \Vert_{H^1}^{2^{r+1}}\bigg)^2 \\
&+ \frac{3}{2} C(r) \bigg(\sum\limits_{l=1}^{r+1}|\Delta_n W|^{2^l}\bigg)(1+ \Vert v^{n-1} \Vert_{H^1}^{2^{r+1}})\Vert v^{n-1} \Vert_{H^1}^{2^{r+1}} \\
&+ 2|f_{n-1}^r|^2 |\Delta_n W|^2 + \frac{1}{8} \bigg(\Vert  v^n \Vert_{H^1}^{2^{r+1}} - \Vert  v^{n-1} \Vert_{H^1}^{2^{r+1}}\bigg)^2 + \frac{3}{2} f_{n-1}^r \Delta_n W \Vert v^{n-1} \Vert_{H^1}^{2^{r+1}} \\
&\leq 4C(r)^2(r+1)^2 \bigg(\sum\limits_{l=1}^{r+1}|\Delta_n W|^{2^{l+1}}\bigg)(1+ \Vert v^{n-1} \Vert_{H^1}^{2^{r+2}}) +  \frac{1}{4} \bigg(\Vert  v^n \Vert_{H^1}^{2^{r+1}} - \Vert  v^{n-1} \Vert_{H^1}^{2^{r+1}}\bigg)^2 \\
&+ 3C(r) \bigg(\sum\limits_{l=1}^{r+1}|\Delta_n W|^{2^l}\bigg)(1+ \Vert v^{n-1} \Vert_{H^1}^{2^{r+2}}) + 2\tilde{C}(r)(1+ \Vert v^{n-1} \Vert_{H^1}^{2^{r+2}})|\Delta_n W|^2 + \frac{3}{2} f_{n-1}^r \Vert  v^{n-1} \Vert_{H^1}^{2^{r+1}} \Delta_n W \\
&\leq \big(4(r+1)^2C(r)^2 + 3C(r) + 2\tilde{C}(r)\big)\bigg(\sum\limits_{l=1}^{r+2}|\Delta_n W|^{2^l}\bigg)(1+ \Vert v^{n-1} \Vert_{H^1}^{2^{r+2}}) \\
&+ \frac{3}{4} f_{n-1}^{r+1} \Delta_n W + \frac{1}{4} \bigg(\Vert  v^n \Vert_{H^1}^{2^{r+1}} - \Vert  v^{n-1} \Vert_{H^1}^{2^{r+1}}\bigg)^2
\end{align*}
$\mathbb{P}$-a.s. in $\Omega$, where $f_{n-1}^{r+1}:= 2f_{n-1}^r \Vert v^{n-1} \Vert_{H^1}^{2^{r+1}}$. 
Hence, subtracting $\frac{1}{4} \bigg(\Vert  v^n \Vert_{H^1}^{2^{r+1}} - \Vert  v^{n-1} \Vert_{H^1}^{2^{r+1}}\bigg)^2$ yields
\begin{align*}
&\frac{3}{4} \bigg(\Vert  v^n \Vert_{H^1}^{2^{r+2}} - \Vert v^{n-1} \Vert_{H^1}^{2^{r+2}}\bigg) \\
&+  \bigg[\prod\limits_{l=1}^{r+1} \Big[\Vert v^n \Vert_{H^1}^{2^l} + \frac{1}{2} \Vert v^{n-1} \Vert_{H^1}^{2^l} \Big] \times \Big( \frac{1}{2} \Vert v^n - v^{n-1} \Vert_{H^1}^2 + \tau \Vert \Delta v^n \Vert_2^2 \Big)\bigg] \notag \\
&\leq \frac{3}{4} C(r+1) \bigg(\sum\limits_{l=1}^{r+2}|\Delta_n W|^{2^l}\bigg)(1+ \Vert v^{n-1} \Vert_{H^1}^{2^{r+2}}) + \frac{3}{4} f_{n-1}^{r+1} \Delta_n W
\end{align*}
$\mathbb{P}$-a.s. in $\Omega$, where $C(r+1):= \frac{4}{3} \big(4(r+1)C(r)^2 + 3C(r) + 2\tilde{C}(r)\big)$. Multiplying with $\frac{4}{3}$ and discarding nonnegative terms on the left-hand side yields \eqref{eq 020224_01} for $r+1$.\\
Now, we have $\E \bigg(\sum\limits_{l=1}^{r+1}|\Delta_n W|^{2^l}\bigg)= \sum\limits_{l=1}^{r+1} \tau^{2^{l-1}} (2l-1)!! \leq \tilde{C}(T, r) \tau$ for some constant $\tilde{C}(T, r)>0$, where
\begin{align*}
k!!=
\begin{cases}
k \cdot (k-2) \cdot ... \cdot 2, &k \in \N ~\text{even}, \\
k \cdot (k-2) \cdot ... \cdot 1, &k \in \N ~\text{odd}
\end{cases}
\end{align*}
and $\E (f_{n-1}^r \Delta_n W) = (\E f_{n-1}^r)(\E \Delta_n W)=0$. Therefore, applying the expectation in \eqref{eq 020224_01} and setting $\overline{C}(T,r):= C(r) \tilde{C}(T,r)$ yields
\begin{align*}
&\E \Vert v^n \Vert_{H^1}^{2^{r+1}} - \E \Vert v^{n-1} \Vert_{H^1}^{2^{r+1}} \\
&+ \E \bigg[\prod\limits_{l=1}^r \Big[\Vert v^n \Vert_{H^1}^{2^l} + \frac{1}{2} \Vert v^{n-1} \Vert_{H^1}^{2^l} \Big] \times \Big( \frac{1}{2} \Vert v^n - v^{n-1} \Vert_{H^1}^2 + \tau \Vert \Delta v^n \Vert_2^2 \Big)\bigg] \\
&\leq \overline{C}(T,r) \tau + \overline{C}(T,r) \tau \E \Vert v^{n-1} \Vert_{H^1}^{2^{r+1}}
\end{align*}
for all $r \in \N_0$. Now summing over $n=1,\ldots, m$ for arbitrary $m\in \{1,\ldots,N\}$, applying Lemma \ref{Gronwall} and taking the supremum yields \eqref{eq 240312_1}. 
\end{proof}
\begin{rem}
Let us remark that $\Delta v^n = (1/\tau) (v^n - v^{n-1} - g(v^{n-1}) \Delta_nW)$ is $(\mathcal{F}_{t_n})$-measurable with values in $H^1(\Lambda)$.
\end{rem}
\begin{lem}\label{Lemma 250923_03}
Let $v^0 \in L^2(\Omega; H^2(\Lambda))$ be $\mathcal{F}_0$-measurable. Then there exists a constant $K_3>0$ such that for any $N \in \mathbb{N}$ we have
\begin{align*}
\sup_{n \in \{1,...,N\}} \mathbb{E} \Vert \Delta v^n \Vert_2^2 + \sum\limits_{n=1}^N \mathbb{E} \Vert \Delta(v^n - v^{n-1}) \Vert_2^2 + \tau \sum\limits_{n=1}^N \mathbb{E} \Vert \nabla \Delta v^n \Vert_2^2 \leq K_3.
\end{align*}
\end{lem}
\begin{proof}
Applying $\nabla$ to \eqref{ES} and testing with $-\nabla \Delta v^n$ yields
\begin{align*}
&\frac{1}{2} \bigg( \Vert \Delta v^n \Vert_2^2 - \Vert \Delta v^{n-1} \Vert_2^2 + \Vert \Delta v^n - \Delta v^{n-1} \Vert_2^2 \bigg) + \tau \Vert \nabla \Delta v^n \Vert_2^2 \\
&= (\Delta g(v^{n-1}) \Delta_n W, \Delta v^n)_2.
\end{align*}
Taking the expectation, using the $\mathcal{F}_{t_{n-1}}$-measurability of $\Delta v^{n-1}$, applying Young's inequality and It\^{o} isometry yields
\begin{align}\label{eq 250923_06}
&\frac{1}{2} \mathbb{E}\bigg( \Vert \Delta v^n \Vert_2^2 - \Vert \Delta v^{n-1} \Vert_2^2 + \Vert \Delta v^n - \Delta v^{n-1} \Vert_2^2 \bigg) + \tau \mathbb{E}\Vert \nabla \Delta v^n \Vert_2^2 \notag \\
&= \mathbb{E}(\Delta g(v^{n-1}) \Delta_n W, \Delta (v^n- v^{n-1}))_2 \\
&\leq \frac{1}{4} \mathbb{E} \Vert \Delta v^n - \Delta v^{n-1} \Vert_2^2 + \tau \mathbb{E} \Vert \Delta g(v^{n-1})\Vert_2^2. \notag
\end{align}
Gagliardo-Nirenberg's inequality, Lemma \ref{Theorem 110124_01} and the boundedness of $g'$ and $g''$ yield the existence of constants $C_1,C_2,C_3,C_4>0$ such that
\begin{align}\label{eq 250923_07} 
\begin{aligned}
&\Vert \Delta g(v^{n-1})\Vert_2^2\\
&\leq \begin{cases} C_1 \Vert \Delta v^{n-1} \Vert_2^2  + C_2 \Vert \Delta v^{n-1} \Vert_2^2 \Vert \nabla v^{n-1} \Vert_2^2 + C_3 \Vert \nabla v^{n-1} \Vert_2^4 + C_4 \Vert v^{n-1} \Vert_2^4, &d=2, \\
C_1 \Vert \Delta v^{n-1} \Vert_2^2  + C_2 \Vert \nabla \Delta v^{n-1} \Vert_2^{3/2} \Vert \nabla v^{n-1} \Vert_2^{5/2} + C_3 \Vert \nabla v^{n-1} \Vert_2^4, &d=3.
\end{cases}
\end{aligned}
\end{align}
Summing over $n=1,...,m$ for any $m \in \{1,...,N\}$ in \eqref{eq 250923_06} and using \eqref{eq 250923_07} in the case $d=2$ yields
\begin{align}\label{eq 250923_08}
&\frac{1}{2} \mathbb{E}\Vert \Delta v^m \Vert_2^2 + \frac{1}{4} \sum\limits_{n=1}^m \mathbb{E} \Vert \Delta v^n - \Delta v^{n-1} \Vert_2^2 + \tau \sum\limits_{n=1}^m  \mathbb{E}\Vert \nabla \Delta v^n \Vert_2^2 \\
&\leq \mathbb{E} \Vert \Delta v^0 \Vert_2^2 + \tau  \sum\limits_{n=1}^m \mathbb{E} \bigg(C_1 \Vert \Delta v^{n-1} \Vert_2^2  + C_2 \Vert \Delta v^{n-1} \Vert_2^2 \Vert \nabla v^{n-1} \Vert_2^2 + C_3 \Vert \nabla v^{n-1} \Vert_2^4 + C_4 \Vert v^{n-1} \Vert_2^4 \bigg) . \notag
\end{align}
In the case $d=3$ we obtain
\begin{align*}
&\frac{1}{2} \mathbb{E}\Vert \Delta v^m \Vert_2^2 + \frac{1}{4} \sum\limits_{n=1}^m \mathbb{E} \Vert \Delta v^n - \Delta v^{n-1} \Vert_2^2 + \tau \sum\limits_{n=1}^m  \mathbb{E}\Vert \nabla \Delta v^n \Vert_2^2 \\
&\leq \mathbb{E} \Vert \Delta v^0 \Vert_2^2 + \tau  \sum\limits_{n=1}^m \mathbb{E} \bigg(C_1 \Vert \Delta v^{n-1} \Vert_2^2  + C_2 \Vert \nabla \Delta v^{n-1} \Vert_2^{5/2} \Vert \nabla v^{n-1} \Vert_2^{3/2} + C_3 \Vert \nabla v^{n-1} \Vert_2^4 \bigg) \\
&\leq \mathbb{E} \Vert \Delta v^0 \Vert_2^2  + \tau  \sum\limits_{n=1}^m \mathbb{E} \bigg(C_1 \Vert \Delta v^{n-1} \Vert_2^2  + \frac{1}{2} \Vert \nabla \Delta v^{n-1} \Vert_2^2 + \tilde{C}_2 \Vert \nabla v^{n-1} \Vert_2^{10} + C_3 \Vert \nabla v^{n-1} \Vert_2^4 \bigg).
\end{align*}
Hence we have
\begin{align}\label{eq 141223_01}
&\frac{1}{2} \mathbb{E}\Vert \Delta v^m \Vert_2^2 + \frac{1}{4} \sum\limits_{n=1}^m \mathbb{E} \Vert \Delta v^n - \Delta v^{n-1} \Vert_2^2 + \frac{\tau}{2} \sum\limits_{n=1}^m  \mathbb{E}\Vert \nabla \Delta v^n \Vert_2^2 \\
&\leq \mathbb{E} \Vert \Delta v^0 \Vert_2^2 + \frac{\tau}{2} \E \Vert \nabla \Delta u_0 \Vert_2^2 + \tau  \sum\limits_{n=1}^m \mathbb{E} \bigg(C_1 \Vert \Delta v^{n-1} \Vert_2^2 + \tilde{C}_2 \Vert \nabla v^{n-1} \Vert_2^{10} + C_3 \Vert \nabla v^{n-1} \Vert_2^4 \bigg). \notag
\end{align}
Applying Lemma \ref{Lemma 250923_02} and Lemma \ref{Gronwall} in both cases $d=2,3$, we may conclude the existence of a constant $\tilde{K}_3>0$ such that
\begin{align*}
\sup_{n \in \{1,...,N\}} \mathbb{E} \Vert \Delta v^n \Vert_2^2 \leq \tilde{K}_3.
\end{align*}
Now, using this estimate in \eqref{eq 250923_08} or \eqref{eq 141223_01} respectively with $m=N$ yields the assertion.
\end{proof}
\section{Stability estimates for the stochastic heat equation}
\subsection{A regularity result}
\begin{prop}\label{240312_prop1}
Let $(R1)$ and $(R2)$ be satisfied. Then, the unique variational solution $u$ to \eqref{equation} has the additional regularity $u\in L^2(\Omega;\mathcal{C}([0,T];H^2(\Lambda)))$ and satisfies the weak homogeneous Neumann boundary condition.
\end{prop}
\begin{rem}
Additionally we make sure that $\Delta u(t)$, the Neumann-Laplacian of $u$ at time $t \in [0,T]$, is an element of $L^2(\Lambda)$. Actually, this is the case if and only if $\Delta u(t)$ satisfies the weak homogeneous Neumann boundary condition for all $t\in [0,T]$.
\end{rem}
\begin{proof}
For $N\in\mathbb{N}$ we define the piecewise constant process
\[u_N(t)=\sum_{n=1}^N u_{n-1}\mathds{1}_{[t_{n-1}, t_n)}(t), \ u_N(T)=u_N \]
where, for $n=1,\ldots,N$, $u_n$ is the unique $\mathcal{F}_{t_n}$-measurable solution to \eqref{ES} starting at the $\mathcal{F}_0$-measurable initial datum $u_0 \in L^2(\Omega; H^2(\Lambda))$. Consequently, $u_N(t)\in H^2(\Lambda)$ $\mathds{P}$-a.s. in $\Omega$ for all $t\in [0,T]$. Repeating the proofs of Theorem 2.1 and Proposition 2.5 in \cite{BBBLV17} with $\psi_{\varepsilon} = w_s = f = 0$,
it follows that the semi-implicit Euler approximations $(u_N)_{N\in\mathbb{N}}$ converge for $N\rightarrow\infty$ towards the unique variational solution $u$ to \eqref{equation} in $L^2(\Omega;L^2(0,T;L^2(\Lambda)))$.
Thanks to \eqref{eq 041023_02},
\begin{align*}
\erww{\Vert u_N(t)\Vert^2_{H^2}}\leq C\left(\erww{\Vert\Delta u_N(t)\Vert^2_2}+\erww{\Vert u_N(t)\Vert^2_2}\right) 
\end{align*}
for all $t\in [0,T]$. By Lemma \ref{Lemma 250923_01} and Lemma \ref{Lemma 250923_03} it follows that the quantities in expectation on the right-hand side of the above equation are uniformly bounded by positive constants $K_1$ and $K_3$, respectively, which do not depend on $t\in [0,T]$ and on $N\in\mathbb{N}$. Consequently, the sequence $(u_N)_{N\in\mathbb{N}}$ is in particular bounded in $L^2(\Omega;L^2(0,T;H^2(\Lambda)))$. Hence, up to a not relabeled subsequence, $(u_N)_{N\in\mathbb{N}}$ converges weakly to $u$ in $L^2(\Omega\times (0,T);H^2(\Lambda)))$. In this way we also get the additional information that $u$ is progressively measurable with values in $H^2(\Lambda)$. Now, according to \cite{LR}, Theorem 4.2.5 with $V=H=H^2(\Lambda)$, we get $u\in L^2(\Omega;\mathcal{C}([0,T];H^2(\Lambda)))$. We make sure that $\Delta u(t)$, the Neumann-Laplacian of $u$ at time $t \in [0,T]$, is an element of $L^2(\Lambda)$ and satisfies the weak homogeneous Neumann boundary condition.
Now, equality \eqref{240311_eq1} yields that $\mathbb{P}$-a.s., for all $t \in [0,T]$ we have $ \int_0^t \Delta u(s) \, ds\in L^2(\Lambda)$.
This yields
\begin{align}\label{270324 eq2}
\begin{aligned}
\bigg( \int_0^t \Delta u(s) \, ds , v \bigg)_2 &= \int_0^t (\Delta u(s), v)_2 \, ds = - \int_0^t \int_{\Lambda} \nabla u(s) \nabla v \, dx \, ds \\
&= \int_0^t \bigg( \int_{\Lambda} (\Delta_K u(s)) v \, dx - \int_{\partial \Lambda} v \nabla u(s) \cdot \mathbf{n} \, dS \bigg) ds \\
&= \bigg( \int_0^t (\Delta_K u(s)),  v \bigg)_2 - \int_0^t \int_{\partial \Lambda} v \nabla u(s) \cdot \mathbf{n} \, dS \, ds 
\end{aligned}
\end{align}
for all $v \in H^1(\Lambda)$, where $\Delta_K w = \sum\limits_{i=1}^d \partial_{ii} w \in L^2(\Lambda)$ denotes the classical Laplace for $w \in H^2(\Lambda)$. For $v \in C_c^{\infty}(\Lambda)$, the boundary term in \eqref{270324 eq2} vanishes and we have
\begin{align}\label{270324 eq3}
\bigg( \int_0^t \Delta u(s) \, ds , v \bigg)_2= \bigg( \int_0^t (\Delta_K u(s)),  v \bigg)_2.
\end{align}
Since $C_c^{\infty}(\Lambda)$ is dense in $L^2(\Lambda)$, equality \eqref{270324 eq3} holds true for all $v \in L^2(\Lambda)$, especially for all $v \in H^1(\Lambda)$. Together with \eqref{270324 eq2} we obtain $\mathbb{P}$-a.s., for all $t \in [0,T]$ and all $v \in H^1(\Lambda)$
\begin{align*}
\int_0^t \int_{\partial \Lambda} v \nabla u(s) \cdot \mathbf{n} \, dS \, ds =0.
\end{align*}
Since $\nabla u $ is continuous with values in $H^1(\Lambda)$, we get $\mathbb{P}$-a.s., for all $t \in [0,T]$ and all $v \in H^1(\Lambda)$
\begin{align*}
\int_{\partial \Lambda} v \nabla u(t) \cdot \mathbf{n} \, dS =0.
\end{align*}
\end{proof}
\subsection{Stability estimates}
For the following results, the regularity $u\in L^2(\Omega;\mathcal{C}([0,T];H^2(\Lambda)))$ of the variational solution to the heat equation is crucial. It is provided by Proposition \ref{240312_prop1}, when $(R1)$ and $(R2)$ are satisfied.
\begin{lem}\label{Lemma 220923_01}
Let (R1) and (R2) be satisfied. Then, there exists a constant 
\[K_4=K_4(T,\Lambda,\Vert u \Vert_{L^2(\Omega; \mathcal{C}([0,T]; H^2(\Lambda)))}, g)>0\] such that for all $s,t \in [0,T]$:
\begin{align*}
\mathbb{E} \Vert u(t) - u(s) \Vert_{L^2(\Lambda)}^2 \leq K_4 |t-s|.
\end{align*}
\end{lem}
\begin{proof}
We fix $s \in [0,T]$. Then, for any $t \in [0,T]$, we have
\begin{align}\label{Eq L2.1}
v(t):= u(t) - u(s) = \int_s^t \Delta u(r) \, dr + \int_s^t g(u(r)) \, dW.
\end{align}
Applying It\^{o}'s formula with $\frac{1}{2} \Vert \cdot \Vert_2^2$ to $v$ we get
\begin{align*}
\frac{1}{2} \Vert v(t) \Vert_2^2 = \int_s^t (\Delta u(r), v(r))_2 \,dr + \int_s^t (v(r), g(u(r)) \, dW)_2 + \frac{1}{2} \int_s^t \Vert g(u(r)) \Vert_2^2 \, dr.
\end{align*}
We apply the expectation on both sides of the equality. Then, we use Young's inequality and the growth condition of $g$ to obtain:
\begin{align*}
\frac{1}{2} \mathbb{E} \Vert v(t) \Vert_2^2 \leq &\frac{1}{2} \mathbb{E} \int_s^t \Vert \Delta u(r) \Vert_2^2 \, dr +  \frac{1}{2} \mathbb{E} \int_s^t \Vert v(r)\Vert_2^2 \,dr +\frac{1}{2} C_g \mathbb{E}\int_s^t |\Lambda| + \Vert u(r) \Vert_2^2 \, dr \\
&\leq \frac{1}{2}\Vert \Delta u \Vert_{L^2(\Omega, \mathcal{C}([0,T]; L^2(\Lambda)))}^2 |t-s|+  \frac{1}{2} \mathbb{E} \int_s^t \Vert v(r)\Vert_2^2 \,dr +\frac{1}{2} C_g |\Lambda| |t-s| \\
&+ \frac{1}{2} C_g \Vert u\Vert_{L^2(\Omega, \mathcal{C}([0,T]; L^2(\Lambda)))}^2 |t-s| \\
&= K_4 |t-s| + \frac{1}{2} \mathbb{E} \int_s^t \Vert v(r)\Vert_2^2 \,dr,
\end{align*}
where $K_4=K_4(T,\Lambda,\Vert u \Vert_{L^2(\Omega; \mathcal{C}([0,T]; H^2(\Lambda)))}, g)>0$. Now, Gronwall's lemma yields the assertion.
\end{proof}
\begin{lem}\label{Lemma 220923_02}
Let (R1) and (R2) be satisfied. Then, there exists a constant 
\[K_5=K_5(T,\Lambda,\Vert u \Vert_{L^2(\Omega; \mathcal{C}([0,T]; H^2(\Lambda)))}, \Vert g' \Vert_{\infty})>0\] 
such that for all $s,t \in [0,T]$:
\begin{align*}
\mathbb{E} \Vert \nabla(u(t) - u(s)) \Vert_{L^2(\Lambda)}^2 \leq K_5 |t-s|.
\end{align*}
\end{lem}
\begin{proof}
We fix $s \in [0,T]$. Then, for any $t \in [0,T]$, we set $v(t):= u(t) - u(s)$. Now, applying It\^{o}'s formula with $\frac{1}{2} \Vert \nabla \cdot \Vert_2^2$ to \eqref{Eq L2.1} we obtain
\begin{align*}
\frac{1}{2} \Vert \nabla v(t) \Vert_2^2 = - \int_s^t (\Delta u(r), \Delta v(r))_2 \,dr - \int_s^t (\Delta v(r), g(u(r)) \, dW)_2 - \frac{1}{2} \int_s^t (g(u(r)), \Delta g(u(r))_2 \, dr.
\end{align*}
Taking the expectation on both sides of the equality, applying Gauss' theorem and Young's inequality yields
\begin{align*}
\frac{1}{2} \mathbb{E} \Vert \nabla v(t) \Vert_2^2 = &- \mathbb{E} \int_s^t (\Delta u(r), \Delta u(r) - \Delta u(s))_2 \,dr + \frac{1}{2} \mathbb{E} \int_s^t \Vert \nabla g(u(r)) \Vert_2^2 \, dr \\
&\leq \mathbb{E} \int_s^t (\Delta u(r), \Delta u(s))_2 \,dr -  \mathbb{E} \int_s^t \Vert \Delta u(r) \Vert_2^2 \, dr+ \frac{1}{2} \Vert g' \Vert_{\infty}^2 \mathbb{E} \int_s^t \Vert \nabla u(r) \Vert_2^2 \, dr \\
&\leq -  \frac{1}{2} \mathbb{E} \int_s^t \Vert \Delta u(r) \Vert_2^2 \, dr + \frac{1}{2} \Vert \Delta u \Vert_{L^2(\Omega; \mathcal{C}([0,T]; L^2(\Lambda)))}^2 |t-s| \\
&+ \frac{1}{2} \Vert g' \Vert_{\infty}^2 \Vert \nabla u \Vert_{L^2(\Omega; \mathcal{C}([0,T]; L^2(\Lambda)))}^2 |t-s| \\
&\leq K_5 |t-s|,
\end{align*}
where $K_5=K_5(T,\Lambda,\Vert u \Vert_{L^2(\Omega; \mathcal{C}([0,T]; H^2(\Lambda)))}, \Vert g' \Vert_{\infty})>0$.
\end{proof}
\section{Convergence rates}\label{Convergence rates}
Now, let us assume that $(R1)$ and $(R2)$ are satisfied and $u$ is a solution to \eqref{equation}. Then, by Proposition \ref{240312_prop1}, $u$ is in $L^2(\Omega; \mathcal{C}([0,T]; H^2(\Lambda)))$ and satisfies the homogeneous Neumann boundary condition in the weak sense.
\begin{lem}\label{Lemma 220923_03}
For $n=1,...,N$ let $v^n$ be a solution to \eqref{ES} for $v^0=u_0$. Then there exists a constant $K_6>0$ not depending on $n$ and $N$ such that
\begin{align*}
\sup_{n\in \{1,...,N\}} \mathbb{E} \Vert u(t_n) - v^n \Vert_2^2 + \tau \sum\limits_{n=1}^N \mathbb{E} \Vert \nabla(u(t_n) - v^n) \Vert_2^2 \leq K_6 \tau.
\end{align*}
\end{lem}
\begin{proof}
We set $e^n:= u(t_n) - v^n$. Then, we have
\begin{align*}
e^n - e^{n-1} = &\int_{t_{n-1}}^{t_n} \Delta (u(s)- v^n) \, ds + \int_{t_{n-1}}^{t_n}g(u(s)) - g(v^{n-1})\, dW\\
&= \int_{t_{n-1}}^{t_n} \Delta (u(s)- u(t_n)) \, ds  + \int_{t_{n-1}}^{t_n} \Delta e^n \, ds\\
&+ \int_{t_{n-1}}^{t_n}g(u(s)) - g(u(t_{n-1}))\, dW + \int_{t_{n-1}}^{t_n}g(u(t_{n-1})) - g(v^{n-1})\, dW.
\end{align*}
Multiplying with $e^n$ and integrating over $\Lambda$ yields
\begin{align*}
\frac{1}{2} &\bigg( \Vert e^n \Vert_2^2 - \Vert e^{n-1} \Vert_2^2 + \Vert e^n - e^{n-1} \Vert_2^2\bigg) = (e^n, e^n - e^{n-1})_2 \\
&= \int_{t_{n-1}}^{t_n} \int_{\Lambda} \Delta (u(s)- u(t_n)) e^n \, ds  + \int_{t_{n-1}}^{t_n} \int_{\Lambda} \Delta e^n \cdot e^n\, ds\\
&+ \int_{t_{n-1}}^{t_n} \int_{\Lambda}(g(u(s)) - g(u(t_{n-1}))) e^n\, dW + \int_{t_{n-1}}^{t_n} \int_{\Lambda} (g(u(t_{n-1})) - g(v^{n-1})) e^n\, dW.
\end{align*}
Taking expectation we obtain the equality
\begin{align}\label{eq 220923_01}
\frac{1}{2} \mathbb{E}\bigg( \Vert e^n \Vert_2^2 - \Vert e^{n-1} \Vert_2^2 + \Vert e^n - e^{n-1} \Vert_2^2\bigg) = I + II + III + IV,
\end{align}
where
\begin{align*}
I&= \mathbb{E} \int_{t_{n-1}}^{t_n} \int_{\Lambda} \Delta (u(s)- u(t_n)) e^n \, ds, \\
II &= \mathbb{E} \int_{t_{n-1}}^{t_n} \int_{\Lambda} \Delta e^n \cdot e^n\, ds, \\
III &= \mathbb{E} \int_{t_{n-1}}^{t_n} \int_{\Lambda}(g(u(s)) - g(u(t_{n-1}))) (e^n- e^{n-1}) \, dW, \\
IV&= \mathbb{E} \int_{t_{n-1}}^{t_n} \int_{\Lambda} (g(u(t_{n-1})) - g(v^{n-1})) (e^n-e^{n-1})\, dW.
\end{align*}
Young's inequality, Lemma \ref{Lemma 220923_01}, Lemma \ref{Lemma 220923_02} and It\^{o} isometry yield
\begin{align*}
| I | \leq &\bigg| \mathbb{E} \int_{t_{n-1}}^{t_n} \int_{\Lambda} \nabla (u(s)- u(t_n)) \nabla e^n \, ds \bigg| \leq \frac{1}{2}\mathbb{E} \int_{t_{n-1}}^{t_n} \Vert \nabla (u(s)- u(t_n)) \Vert_2^2 \, ds + \frac{1}{2} \mathbb{E} \int_{t_{n-1}}^{t_n}  \Vert \nabla e^n \Vert_2^2 \, ds\\
&\leq \frac{K_5}{2} \tau^2 + \frac{\tau}{2} \mathbb{E} \Vert \nabla e^n \Vert_2^2, \\
II &= -\tau \mathbb{E} \Vert \nabla e^n \Vert_2^2, \\
III &\leq 2\mathbb{E} \int_{t_{n-1}}^{t_n} \Vert g(u(s)) - g(u(t_{n-1})) \Vert_2^2 \, ds + \frac{1}{8} \mathbb{E} \Vert e^n - e^{n-1} \Vert_2^2 \\
&\leq 2\Vert g' \Vert_{\infty}^2 \mathbb{E} \int_{t_{n-1}}^{t_n} \Vert u(s) - u(t_{n-1}) \Vert_2^2 \, ds + \frac{1}{8} \mathbb{E} \Vert e^n - e^{n-1} \Vert_2^2 \\
&\leq 2\Vert g' \Vert_{\infty}^2 K_4 \tau^2 + \frac{1}{8} \mathbb{E} \Vert e^n - e^{n-1} \Vert_2^2, \\
IV &\leq 2\mathbb{E} \int_{t_{n-1}}^{t_n} \Vert g(u(t_{n-1})) - g(v^{n-1}) \Vert_2^2 \, ds + \frac{1}{8} \mathbb{E} \Vert e^n - e^{n-1} \Vert_2^2 \\
&\leq 2\Vert g' \Vert_{\infty}^2 \tau \mathbb{E} \Vert e^{n-1} \Vert_2^2 + \frac{1}{8} \mathbb{E} \Vert e^n - e^{n-1} \Vert_2^2.
\end{align*}
Summing over $n=1,...,m$, $m \in \{1,...,N\}$, equality \eqref{eq 220923_01} and the previous estimates yield
\begin{align*}
&\frac{1}{2} \mathbb{E} \Vert e^m \Vert_2^2 + \frac{1}{4} \sum\limits_{n=1}^m \mathbb{E} \Vert e^n - e^{n-1} \Vert_2^2 + \frac{\tau}{2} \sum\limits_{n=1}^m \mathbb{E} \Vert \nabla e^n \Vert_2^2\\
&\leq \left(2\Vert g' \Vert_{\infty}^2 K_4 T + \frac{K_5}{2}\right) \tau + 2\Vert g' \Vert_{\infty}^2 \tau \sum\limits_{n=1}^m \mathbb{E} \Vert e^{n-1} \Vert_2^2.
\end{align*}
Lemma \ref{Gronwall} yields that there exists $\tilde{K}_6>0$ such that 
\begin{align*}
\sup_{n=1,...,N} \mathbb{E} \Vert e^n \Vert_2^2 \leq \tilde{K}_6 \tau
\end{align*}
and therefore the assertion holds.
\end{proof}
\begin{prop}\label{281123_01}
Let $w \in H^2(\Lambda)$ and $\mathcal{T}$ an admissible mesh. Then, there exists a unique solution $(\tilde{w}_K)_{K \in \mathcal{T}}$ to the following discrete problem:\\
Find $(\tilde{w}_K)_{K \in \mathcal{T}}$ such that 
\begin{align*}
\sum\limits_{K \in \mathcal{T}} m_K \tilde{w}_K = \int_{\Lambda} w \, dx
\end{align*}
and
\begin{align*}
\sum\limits_{\sigma \in \mathcal{E}_K^{int}} \frac{m_{\sigma}}{d_{K|L}} (\tilde{w}_K - \tilde{w}_L) = - \int_{K} \Delta w \, dx, ~~~\forall K \in \mathcal{T}.
\end{align*}
\end{prop}
\begin{proof}
See e.g. \cite{Flore21}, Definition 6 and Remark 2.
\end{proof}
\begin{defn}\label{Definition parabolic projection}
Let $\mathcal{T}$ be an admissible mesh and $w \in H^2(\Lambda)$. The elliptic projection of $w$ on $\Tau$ is given by
\begin{align*}
\tilde{w}:= \sum\limits_{K \in \mathcal{T}} \tilde{w}_K \mathds{1}_K,
\end{align*}
where $(\tilde{w}_K)_{K \in \mathcal{T}}$ are given by Proposition \ref{281123_01}. The centered projection of $w$ on $\Tau$ is given by
\begin{align*}
\hat{w}:= \sum\limits_{K \in \mathcal{T}} w(x_K) \mathds{1}_K.
\end{align*}
\end{defn}
\begin{lem}\label{Lemma 260923_01}
Let $\Tau_h$ be an admissible mesh with $\operatorname{size}(\mathcal{T}_h)=h$ and $w$ a random variable with values in $H^2(\Lambda)$.\\
Moreover, let $\tilde{w}_h$ be the elliptic projection of $w$ a.s. in $\Omega$. Then there exists a constant $K_7>0$ only depending on $\Lambda$ and $\operatorname{reg}(\Tau_h)$ such that
\begin{align*}
\Vert w - \tilde{w}_h \Vert_2^2 \leq K_7 h^2 \Vert w \Vert_{H^2(\Lambda)}^2 ~~~ \mathbb{P}\text{-a.s. in}~\Omega.
\end{align*}
Furthermore, let $\hat{w}_h$ be the centered projection of $w$ a.s. in $\Omega$. Then there exists a constant $K_8>0$ only depending on $\Lambda$ and $\operatorname{reg}(\Tau_h)$ such that
\begin{align*}
\Vert \hat{w}_h - \tilde{w}_h \Vert_2^2 \leq K_8 h^2 \Vert w \Vert_{H^2(\Lambda)}^2 ~~~ \mathbb{P}\text{-a.s. in}~\Omega.
\end{align*}
If \eqref{mesh_regularity} is satisfied, $K_7$ and $K_8$ may depend on $\chi$ and the dependence of $K_7$ and $K_8$ on $\operatorname{reg}(\Tau_h)$ can be omitted.
\end{lem}
\begin{proof}
First, from Lemma \ref{Lemma 240124_01} we observe that there exists a constant $\tilde{K}>0$ only depending on $\Lambda$ such that
\begin{align*}
\Vert w- \hat{w}_h \Vert_2^2 \leq \tilde{K} h^2 \Vert w \Vert_{H^2(\Lambda)}^2 ~~~ \mathbb{P}\text{-a.s. in}~\Omega.
\end{align*}
Moreover, by Definition of the elliptic projection we have
\begin{align*}
\int_{\Lambda} \hat{w}_h - \tilde{w}_h \, dx = \int_{\Lambda} \hat{w}_h \, dx - \sum\limits_{K \in \mathcal{T}_h} \tilde{w}_K = \int_{\Lambda} \hat{w}_h - w \, dx.
\end{align*}
Therefore, H\"{o}lder inequality and Lemma \ref{Lemma 240124_01} yield 
\begin{align}\label{250115_01}
\bigg( \int_{\Lambda} \hat{w}_h - \tilde{w}_h \, dx \bigg)^2 \leq |\Lambda| \int_{\Lambda} |\hat{w}_h - w|^2 \, dx \leq |\Lambda|C(2,\Lambda) h^2 \Vert w \Vert_{H^2(\Lambda)}^2 ~~~ \mathbb{P}\text{-a.s. in}~\Omega.
\end{align}
Now, from \eqref{291123_01} and \eqref{250115_01} we have
\begin{align*}
\Vert \hat{w}_h - \tilde{w}_h \Vert_2^2 &\leq C_p | \hat{w}_h - \tilde{w}_h |_{{1, h}}^2 + 2 |\Lambda|^{-1} \bigg( \int_{\Lambda} \hat{w}_h - \tilde{w}_h \, dx \bigg)^2 \\
&\leq C_p | \hat{w}_h - \tilde{w}_h |_{{1, h}}^2 + 2 C(2,\Lambda) h^2 \Vert w \Vert_{H^2(\Lambda)}^2.
\end{align*}
Hence, it is left to show that there exists a constant $\overline{K}>0$ only depending on $\Lambda$ such that
\begin{align*}
| \hat{w}_h - \tilde{w}_h |_{{1, h}}^2 \leq \overline{K} \operatorname{reg}(\Tau_h)^2 h^2 \Vert w \Vert_{H^2(\Lambda)}^2.
\end{align*}
You may obtain this result by proceeding similarly as in \cite{EGH00}, Section 3.2.3.
\end{proof}
\begin{cor}\label{Corollary 260923_01}
For any $N \in \mathbb{N}$ let $v^n$, $n \in \{1,...,N\}$ be the solution to \eqref{ES} and $\tilde{v}_h^n$ the elliptic projection of $v^n$, a.s. in $\Omega$. Then there exists a constant $K_9>0$ such that
\begin{align*}
\sup_{n\in \{1,...,N\}} \mathbb{E} \Vert v^n - \tilde{v}_h^n \Vert_2^2 \leq K_9 h^2.
\end{align*}
Especially, we have
\begin{align*}
\tau \sum\limits_{n=1}^{N} \mathbb{E} \Vert v^n - \tilde{v}_h^n \Vert_2^2 \leq K_9 T h^2.
\end{align*}
\end{cor}
\begin{proof}
From Lemma \ref{Lemma 260923_01} and \eqref{eq 041023_02} it follows that for all $n=1,...,N$ we have
\begin{align*}
\mathbb{E} \Vert v^n - \tilde{v}_h^n \Vert_2^2 \leq K_7 h^2 \mathbb{E} \Vert v^n \Vert_{H^2(\Lambda)}^2 \leq K_7 C h^2 \mathbb{E} (\Vert \Delta v^n \Vert_2^2 + \Vert v^n \Vert_2^2).
\end{align*}
Thus, from Lemma \ref{Lemma 250923_01} and Lemma \ref{Lemma 250923_03} we get
\begin{align*}
\mathbb{E} \Vert v^n - \tilde{v}_h^n \Vert_2^2 \leq K_7 C (K_3 + K_1) h^2.
\end{align*}
Multiplying this inequality with $\tau$ and summing over $n=1,...,N$ yields the result.
\end{proof}

\begin{lem}\label{Lemma 270923_01}
	For any $N \in \mathbb{N}$ let $v^n$, $n \in \{1,...,N\}$ be the solution to \eqref{ES}. Let $\tilde{v}_h^n$ be the elliptic projection of $v^n$. Then there exists a constant $K_{10}>0$ not depending on $N,n,h$ such that
	\begin{align*}
		\sup_{n\in \{1,...,N\}} \mathbb{E} \Vert \tilde{v}_h^n - u_h^n \Vert_2^2 \leq K_{10}(h^2 + \frac{h^2}{\tau}).
	\end{align*}
\end{lem}
\begin{proof}
	We divide \eqref{FVS} by $m_K$ and subtract this equality from \eqref{ES} and get
	\begin{align*}
		&\bigg( v^n - u_K^n - (v^{n-1} - u_K^{n-1}) \bigg) - \tau \bigg( \Delta v^n + \frac{1}{m_K} \sum\limits_{\sigma \in \mathcal{E}_K^{int}} \frac{m_{\sigma}}{d_{K|L}} (u_K^n - u_L^n) \bigg) \\
		&= \bigg( g(v^{n-1}) - g(u_K^{n-1}) \bigg) \Delta_n W ~~~\text{on}~K.
	\end{align*}
	Adding and subtracting $\tilde{v}_K^n - \tilde{v}_K^{n-1}$ in the first large brackets, then multiplying with $\tilde{v}_K^n - u_K^n$ and integrating over $K$ yields
	\begin{align*}
		I_K + II_K + III_K = IV_K,
	\end{align*}
	where using the definition of the elliptic projection, we obtain
	\begin{align*}
		I_K &= \int_K \bigg(v^n - \tilde{v}_K^n - (v^{n-1} - \tilde{v}_K^{n-1})\, \bigg) \, dx \,  ( \tilde{v}_K^n - u_K^n) \\
		II_K &= \frac{1}{2} \bigg( \Vert \tilde{v}_K^n - u_K^n \Vert_{L^2(K)}^2 - \Vert \tilde{v}_K^{n-1} - u_K^{n-1} \Vert_{L^2(K)}^2 + \Vert \tilde{v}_K^n - u_K^n - (\tilde{v}_K^{n-1} - u_K^{n-1}) \Vert_{L^2(K)}^2 \bigg) \\
		III_K &= -\tau\bigg( \int_K \Delta v^n \, dx + \sum\limits_{\sigma \in \mathcal{E}_K^{int}} \frac{m_{\sigma}}{d_{K|L}} (u_K^n - u_L^n)\bigg) (\tilde{v}_K^n - u_K^n) \\
		&= + \tau\bigg( \sum\limits_{\sigma \in \mathcal{E}_K^{int}} \frac{m_{\sigma}}{d_{K|L}} (\tilde{v}_K^n - u_K^n - (\tilde{v}_L^n - u_L^n))\bigg)(\tilde{v}_K^n - u_K^n) \\
		IV_K &= \int_K (g(v^{n-1}) - g(\tilde{v}_K^{n-1})) ( \tilde{v}_K^n - u_K^n)  \Delta_n W \, dx + \int_K (g(\tilde{v}_K^{n-1}) - g(u_K^{n-1})) ( \tilde{v}_K^n - u_K^n)  \Delta_n W \, dx.
	\end{align*}
	Summing over $K \in \mathcal{T}$ and using the abbreviations $\dot{e}_h^n:= \tilde{v}_h^n - u_h^n$ and $\overline{e}_h^n:= v^n - \tilde{v}_h^n$ yield
	\begin{align}\label{eq 260923_01}
		I + II + III= IV,
	\end{align}
	where, rewriting $III$ by using the discrete integration by parts rule (see Remark \ref{DIBP}), we have
	\begin{align*}
		I &= \int_{\Lambda} (\overline{e}_h^n - \overline{e}_h^{n-1}) \dot{e}_h^n= (\overline{e}_h^n - \overline{e}_h^{n-1}, \dot{e}_h^n)_2 \\
		II &= \frac{1}{2} \bigg( \Vert \dot{e}_h^n \Vert_2^2 - \Vert \dot{e}_h^{n-1} \Vert_2^2 + \Vert \dot{e}_h^n - \dot{e}_h^{n-1}\Vert_2^2 \bigg) \\
		III &= \tau \sum\limits_{\sigma \in \mathcal{E}_{int}} \frac{m_{\sigma}}{d_{K|L}} \bigg(\tilde{v}_K^n - u_K^n - (\tilde{v}_L^n - u_L^n)\bigg)^2= \tau |\tilde{v}_h^n - u_h^n|_{1,h}^2 = \tau |\dot{e}_h^n|_{1,h}^2\\
		IV &= \int_{\Lambda} (g(v^{n-1}) - g(\tilde{v}_h^{n-1})) \dot{e}_h^n  \Delta_n W \, dx + \int_{\Lambda} (g(\tilde{v}_h^{n-1}) - g(u_h^{n-1}))\dot{e}_h^n  \Delta_n W \, dx.
	\end{align*}
	Taking the expectation in $IV$ and applying Young's inequality yields
	\begin{align*}
		\mathbb{E}(IV) &\leq 2\tau \mathbb{E} \Vert g(v^{n-1}) - g(\tilde{v}_h^{n-1})\Vert_2^2 + \frac{1}{8} \mathbb{E} \Vert \dot{e}_h^n - \dot{e}_h^{n-1}\Vert_2^2 \\
		&+ 2\tau \mathbb{E} \Vert g(\tilde{v}_h^{n-1}) - g(u_h^{n-1})\Vert_2^2 + \frac{1}{8} \mathbb{E} \Vert \dot{e}_h^n - \dot{e}_h^{n-1}\Vert_2^2.
	\end{align*}
	Hence, taking expectation in \eqref{eq 260923_01} and using the boundedness of $g'$ yields
	\begin{align*}
		&\mathbb{E} (\overline{e}_h^n - \overline{e}_h^{n-1}, \dot{e}_h^n)_2 + \frac{1}{2} \mathbb{E} (\Vert \dot{e}_h^n \Vert_2^2 - \Vert \dot{e}_h^{n-1} \Vert_2^2) + \frac{1}{4}  \mathbb{E} \Vert \dot{e}_h^n - \dot{e}_h^{n-1}\Vert_2^2 + \tau \mathbb{E} |\dot{e}_h^n|_{1,h}^2 \\
		&\leq 2 \Vert g' \Vert_{\infty}^2 \tau \mathbb{E} \Vert \overline{e}_h^{n-1} \Vert_2^2 + 2 \Vert g' \Vert_{\infty}^2 \tau \mathbb{E} \Vert \dot{e}_h^{n-1} \Vert_2^2.
	\end{align*}
	We sum over $n=1,...,m$ for $m \in \{1,...,N\}$ and obtain
	\begin{align}\label{eq 260923_02}
		&\sum\limits_{n=1}^m \mathbb{E} (\overline{e}_h^n - \overline{e}_h^{n-1}, \dot{e}_h^n)_2 + \frac{1}{2} \mathbb{E} \Vert \dot{e}_h^m \Vert_2^2 + \frac{1}{4}  \sum\limits_{n=1}^m \mathbb{E} \Vert \dot{e}_h^n - \dot{e}_h^{n-1}\Vert_2^2 + \tau \sum\limits_{n=1}^m\mathbb{E} |\dot{e}_h^n|_{1,h}^2 \\
		&\leq \frac{1}{2} \mathbb{E} \Vert \dot{e}_h^0 \Vert_2^2 + 2 \Vert g' \Vert_{\infty}^2 \tau \sum\limits_{n=1}^m \mathbb{E} \Vert \overline{e}_h^{n-1} \Vert_2^2 + 2 \Vert g' \Vert_{\infty}^2 \tau\sum\limits_{n=1}^m \mathbb{E} \Vert \dot{e}_h^{n-1} \Vert_2^2. \notag
	\end{align}
	Being $(\hat{u_0})_h$ the centered projection of $u_0=v^0$, Lemma \ref{Lemma 260923_01} and Lemma \ref{Lemma 240124_01} yield
	\begin{align*}
		&\Vert \dot{e}_h^0 \Vert_2^2 = \Vert \tilde{v}_h^0 - u_h^0 \Vert_2^2 \leq 2 \Vert \tilde{v}_h^0 - (\hat{u_0})_h \Vert_2^2 + 2 \Vert (\hat{u_0})_h - u_h^0 \Vert_2^2 \notag \\
		&\leq 2K_8h^2 \Vert u_0 \Vert_{H^2(\Lambda)}^2+  2 \sum\limits_{K \in \mathcal{T}} \int_K |u_0(x_K) - \frac{1}{m_K} \int_K u_0(y) \, dy |^2 \, dx \notag \\
		&=  2K_8h^2 \Vert u_0 \Vert_{H^2(\Lambda)}^2 +  2 \sum\limits_{K \in \mathcal{T}} \frac{1}{m_K} | \int_K u_0(x_K) -  u_0(y) \, dy |^2 \\
		&\leq  2K_8h^2 \Vert u_0 \Vert_{H^2(\Lambda)}^2+  2 \sum\limits_{K \in \mathcal{T}} C(2,\Lambda) h^2 \Vert u_0 \Vert_{H^2(K)}^2 \notag \\
		&= 2(K_8+ 2 C(2,\Lambda)) h^2 \Vert u_0 \Vert_{H^2(\Lambda)}^2 \notag
	\end{align*}
	a.s. in $\Omega$. Hence
	\begin{align}\label{eq 260923_03}
		\E \Vert \dot{e}_h^0 \Vert_2^2 &\leq \E \bigg( 2(K_8+ 2 C(2,\Lambda)) h^2 \Vert u_0 \Vert_{H^2(\Lambda)}^2 \bigg) \\
		&= K_{11} h^2 \notag
	\end{align}
	for $K_{11}:= 2(K_8+ 2 C(2,\Lambda)) \E \Vert u_0 \Vert_{H^2(\Lambda)}^2$.
	We note that $\int_{\Lambda} \overline{e}_h^n \, dx = \int_{\Lambda} v_h^n - \tilde{v}_h^n \, dx= 0$ for $n=0,...,N$, hence $\int_{\Lambda} (\overline{e}_h^n - \overline{e}_h^{n-1}) \cdot c \, dx =0 $ for all $n=1,...,N$ and $c \in \mathbb{R}$ by the definition of $\tilde{v}_h^n$. Then, setting $c = \frac{1}{|\Lambda|} \int_{\Lambda} \dot{e}_h^n \, dx$, using Cauchy-Schwarz inequality and \eqref{291123_01} we get
	\begin{align*}
		|\mathbb{E} (\overline{e}_h^n - \overline{e}_h^{n-1}, \dot{e}_h^n)_2| &= |\mathbb{E}  (\overline{e}_h^n - \overline{e}_h^{n-1}, \dot{e}_h^n - c)_2| \\
  &\leq \erww{\Vert \overline{e}_h^n - \overline{e}_h^{n-1}\Vert_2 \Vert \dot{e}_h^n -c \Vert_2} \\
  &\leq \erww{\Vert \overline{e}_h^n - \overline{e}_h^{n-1}\Vert_2 \bigg( C_p| \dot{e}_h^n -c |^2_{1,h} + 2|\Lambda|^{-1} \big( \int_{\Lambda} \dot{e}_h^n -c \, dx \big)^2 \bigg)^{\frac{1}{2}}} \\
  &= \sqrt{C_p} \erww{\Vert \overline{e}_h^n - \overline{e}_h^{n-1}\Vert_2 | \dot{e}_h^n -c |_{1,h}} \\
  &= \sqrt{C_p} \erww{ \Vert \overline{e}_h^n - \overline{e}_h^{n-1}\Vert_2 | \dot{e}_h^n |_{1,h}}.
	\end{align*}

 
	Now, using Young's inequality, Lemma \ref{Lemma 260923_01}, Lemma \ref{Theorem 110124_01} and the fact that $\tilde{v}_h^n - \tilde{v}_h^{n-1}$ is the elliptic projection of $v^n - v^{n-1}$ we obtain
	\begin{align}\label{eq 260923_04}
		\begin{aligned}
			&|\mathbb{E} (\overline{e}_h^n - \overline{e}_h^{n-1}, \dot{e}_h^n)_2| \\
			&\leq \frac{C_p}{2\tau} \mathbb{E} \Vert \overline{e}_h^n - \overline{e}_h^{n-1}\Vert_2^2 + \frac{\tau}{2} \mathbb{E} | \dot{e}_h^n |_{1,h}^2\\
			&= \frac{C_p}{2\tau} \E \Vert v^n - \tilde{v}_h^n - (v^{n-1} - \tilde{v}_h^{n-1}) \Vert_2^2 + \frac{\tau}{2} \mathbb{E} | \dot{e}_h^n |_{1,h}^2\\
			&= \frac{C_p}{2\tau} \E \Vert (v^n - v^{n-1})- (\tilde{v}_h^n - \tilde{v}_h^{n-1}) \Vert_2^2  + \frac{\tau}{2} \mathbb{E} | \dot{e}_h^n |_{1,h}^2 \\
			&\leq \frac{K_7h^2C_p}{2\tau} \E \Vert v^n - v^{n-1} \Vert_{H^2(\Lambda)}^2 + \frac{\tau}{2} \mathbb{E} | \dot{e}_h^n |_{1,h}^2 \\
			&\leq \frac{6K_7h^2C_p}{\tau} \E \bigg(\Vert v^n - v^{n-1} \Vert_2^2 + \Vert \Delta v^n - \Delta v^{n-1} \Vert_2^2 \bigg) + \frac{\tau}{2} \mathbb{E} | \dot{e}_h^n |_{1,h}^2.
		\end{aligned}
	\end{align}
	Summing over $n=1,...,m$ for $m \in \N$ and using Lemma \ref{Lemma 250923_01}, Lemma \ref{Lemma 250923_03} and inequality \eqref{eq 260923_04} yields
	\begin{align}\label{eq 300424_01}
		\begin{aligned}
			&\sum\limits_{n=1}^m |\mathbb{E} (\overline{e}_h^n - \overline{e}_h^{n-1}, \dot{e}_h^n)_2| \\
			&\leq 6K_7(K_1+ K_3)C_p \frac{h^2}{\tau} + \frac{\tau}{2} \sum\limits_{n=1}^m  \mathbb{E} | \dot{e}_h^n |_{1,h}^2.
		\end{aligned}
	\end{align}
	Now we combine \eqref{eq 260923_02}, \eqref{eq 260923_03}, \eqref{eq 300424_01} and Corollary \ref{Corollary 260923_01} and get
	\begin{align*}
		&\frac{1}{2} \mathbb{E} \Vert \dot{e}_h^m \Vert_2^2 + \frac{1}{4}  \sum\limits_{n=1}^m \mathbb{E} \Vert \dot{e}_h^n - \dot{e}_h^{n-1}\Vert_2^2 + \frac{\tau}{2} \sum\limits_{n=1}^m\mathbb{E} |\dot{e}_h^n|_{1,h}^2 \\
		&\leq \bigg( 6K_7(K_1+ K_3)C_p \frac{1}{\tau} +\frac{K_{11}}{2} + 2 \Vert g' \Vert_{\infty}^2 K_9T \bigg) h^2 + 2 \Vert g' \Vert_{\infty}^2 \tau\sum\limits_{n=1}^m \mathbb{E} \Vert \dot{e}_h^{n-1} \Vert_2^2.
	\end{align*}
	Applying Lemma \ref{Gronwall} finalizes the proof of this lemma.
\end{proof}

\begin{rem}
The constants $K_9$ and $K_{10}$ respectively depend on $K_7$ and this constant may depend on $\operatorname{reg}(\Tau_h)$. Assuming \eqref{mesh_regularity}, we may overcome this dependence.
\end{rem}

\section{Proof of Theorem \ref{main theorem}}
Let $t \in [0,T]$ and $N \in \mathbb{N}$. Then there exists $n \in \{1,...,N-1\}$ such that $t \in [t_{n-1}, t_n)$ or $t \in [t_{N-1}, t_N]$. For $n \in \{1,...,N\}$ we have
\begin{align*}
&\mathbb{E} \Vert u(t) - u_{h,N}^r(t) \Vert_2^2 = \mathbb{E} \Vert u(t) - u_h^n \Vert_2^2\\
&\leq \mathbb{E} \Vert u(t) - u(t_n) + u(t_n) - v^n + v^n - \tilde{v}_h^n + \tilde{v}_h^n - u_h^n \Vert_2^2 \\
&\leq 16 \bigg(\mathbb{E} \Vert u(t) - u(t_n) \Vert_2^2 + \mathbb{E} \Vert u(t_n) - v^n \Vert_2^2 + \mathbb{E} \Vert v^n - \tilde{v}_h^n \Vert_2^2  + \mathbb{E} \Vert \tilde{v}_h^n - u_h^n \Vert_2^2 \bigg),
\end{align*}
where $v^n$ is the solution to \eqref{ES} with initial value $u_0$ and $\tilde{v}_h^n$ is the elliptic projection of $v^n$. Now, Lemma \ref{Lemma 220923_01}, Lemma \ref{Lemma 220923_03}, Corollary \ref{Corollary 260923_01} and Lemma \ref{Lemma 270923_01} yield
\begin{align*}
&\mathbb{E} \Vert u(t) - u_{h,N}^r(t) \Vert_2^2 \\
&\leq 16(K_4 \tau + K_6 \tau + K_9 h^2 + K_{10} (h^2 + \frac{h^2}{\tau})) \\
&\leq \Upsilon (\tau + h^2 + \frac{h^2}{\tau})
\end{align*}
for some constant $\Upsilon >0$ depending on the mesh regularity $\operatorname{reg}(\Tau_h)$ but not depending on $n, N$ and $h$ explicitly. If \eqref{mesh_regularity} is satisfied, $\Upsilon$ may depend on $\chi$ and the dependence of $\Upsilon$ on $\operatorname{reg}(\Tau_h)$ can be omitted. Therefore, taking the supremum over $t \in [0,T]$ on the left hand side and the right hand side of the inequality, we obtain
\begin{align*}
\sup_{t \in [0,T]} \mathbb{E} \Vert u(t) - u_{h,N}^r(t) \Vert_2^2 \leq \Upsilon (\tau + h^2 + \frac{h^2}{\tau}).
\end{align*}
\section{Computational experiments}
In this section we provide some computational experiments in order to validate our analytical result in Theorem \ref{main theorem}. The code for all experiments has been written in Python. The first two experiments, displayed in Tab. \ref{table 1} and Figure \ref{figure 1}, were performed on an Apple M2 Pro with 16GB RAM. The third experiment, displayed in Tab. \ref{table 2} and Figure \ref{figure 2}, was performed on an AMD Epyc 7502 with 1024GB RAM. In all experiments we use $\Lambda =(-1,1)^2$ and $T=1$. We use $u_0(x,y) = \frac{1}{16}(x^4 + 4x^3 - 2x^2 - 12x)(y^4 - \frac{8}{3}y^3 - 2y^2 + 8y)$ as a non-symmetric and non-trivial initial value satisfying the homogeneous Neumann boundary condition. For the noise term $g$ we use $g(u) = \sqrt{1+u^2}$. In order to compute the expectation we use the classical Monte Carlo method with $N_p = 15000$ realizations for the first two experiments and due to the very long calculation time $N_p = 1000$ realizations for the third experiment. We will divide the domain $\Lambda$ into subsquares with equal size which are our control volumes. The parameters $L$ or $L_{max}$ will denote the number of squares in each spatial direction, i.e., the total number of subsquares used for the spatial discretization is $L^2$ or $L_{max}^2$ and the 2-d Lebesgue measure of those subsquares is $4/L^2$ or $4/L_{max}^2$, respectively. Therefore, we have $h = \sqrt{8}/L$ or $h = \sqrt{8}/L_{max{}}$, respectively. In all the experiments, we compute the following squared $L^2$-error
\begin{equation*}
E(L, L_{max}, N, N_{max}) := \frac{1}{N_p} \sum\limits_{l=1}^{N_p} \Vert U_{L_{max}}^{N_{max}}(\omega_l) - U_L^N(\omega_l) \Vert_{L^2(\Lambda)}^2 \approx \mathbb{E} \Vert U_{L_{max}}^{N_{max}} - U_L^N \Vert_{L^2(\Lambda)}^2 
\end{equation*}
where $U_{L_{max}}^{N_{max}}(\omega_l)$, $\omega_l\in\Omega$, represents the finite volume approximation of $u(1, \omega_l)$ using $\tau = 1/N_{max}$ with $N_{max}$ time steps and $L_{max}$ subsquares in each spatial direction. Analogously, $U_L^N(\omega_l)$ represents the finite volume approximation of $u(1, \omega_l)$ using $\tau = 1/N$ with $N$ time steps and $L$ subsquares in each spatial direction. In Tab. \ref{table 1} and Figure \ref{figure 1} we present the results of the first two computational experiments in order to validate the terms $\tau$ and $h^2$ in Theorem \ref{main theorem}. For the first experiment we use $L_{max} = L = 40$, $N_{max} = 10240$ and $N \in \{64, 128, 256, 512, 1024 \}$ and for the second experiment we use $N_{max} = N = 10240$, $L_{max} = 40$ and $L \in \{6, 8, 10, 12, 14\}$. In the third and sixth column of Tab. \ref{table 1}  we present the order of convergence of the squared $L^2$-error compared to the previous time step or spatial step, respectively.
\begin{table}[!h]
\centering
\caption{Squared $L^2$-error with different time steps and different spatial steps}
\begin{tabular}{c|c|c||c|c|c}
\multicolumn{3}{c}{$N_{max}=10240, ~L_{max}=L=40$} & \multicolumn{3}{c}{$N_{max}=N=10240, ~L_{max}=40$} \\
\cmidrule(lr){1-3} \cmidrule(lr){4-6}
$N$ & $E(L, L_{max}, N, N_{max})$ & \text{order in time} & $L$ & $E(L, L_{max}, N, N_{max})$ & \text{order in space} \\
\hline
$64$ & $3.839\cdot 10^{-2}$ & \phantom{-} & $6$ & $1.533\cdot 10^{-4}$ & \phantom{-}  \\
\hline
$128$ & $1.783\cdot 10^{-2}$ & $1.107$  & $8$ & $6.236\cdot 10^{-5}$ & $3.126$ \\
\hline
$256$ & $8.851\cdot 10^{-3}$ & $1.010$ & $10$ & $3.269\cdot 10^{-5}$ & $2.895$\\
\hline
$512$ & $4.122\cdot 10^{-3}$ & $1.103$ & $12$ & $2.234\cdot 10^{-5}$ & $2.088$\\
\hline
$1024$ & $1.971\cdot 10^{-3}$ & $1.064$ & $14$ & $1.585\cdot 10^{-5}$ & $2.224$ \\
\end{tabular}
\label{table 1}
\end{table}
\\
\noindent
The last experiment is presented in Tab. \ref{table 2} and Figure \ref{figure 2}. We tried to validate the term $h^2/ \tau$ in Theorem \ref{main theorem}. To this end, we set $L_{max} = 16$ and $L \in \{4,8\}$ and choose high values for $N_{max}$ and $N$. More precisely, we set $N_{max} = 2^{22} = 4194304$ and $N \in \{2^{14},..., 2^{19}\} = \{16384, 25768, 65536, 131072, 262144 \}$. The result for $L=4$ indicates that there might be some term of $\tau$ with a negative exponent, but the increase of the error between some of the time steps might be caused due to computational inaccuracies. For $L=8$ the graph shows no increase of the error with increasing step size. 
\begin{table}[htbp]
\centering
\caption{High number of time steps with $L_{max}=16$ and $L=4$ or $L=8$, respectively.}
\begin{tabular}{c|c||c|c}
\multicolumn{2}{c}{$L=4$} & \multicolumn{2}{c}{$L=8$} \\
\cmidrule(lr){1-2} \cmidrule(lr){3-4}
$N$ & $E(L, L_{max}, N, N_{max})$ & $N$ & $E(L, L_{max}, N, N_{max})$\\
\hline
$2^{14}$ & $5.994 \cdot 10^{-4}$ & $2^{14}$ & $1.730\cdot 10^{-4}$ \\
\hline
$2^{15}$ & $5.084\cdot 10^{-4}$ & $2^{15}$ & $1.040\cdot 10^{-4}$\\
\hline
$2^{16}$ & $5.277\cdot 10^{-4}$ & $2^{16}$ & $7.794\cdot 10^{-5}$\\
\hline
$2^{17}$ & $5.391\cdot 10^{-4}$ & $2^{17}$ & $6.735\cdot 10^{-5}$\\
\hline
$2^{18}$ & $5.278\cdot 10^{-4}$ & $2^{18}$ & $5.470\cdot 10^{-5}$ \\
\hline
$2^{19}$ & $5.096\cdot 10^{-4}$ & $2^{19}$ & $4.771\cdot 10^{-5}$\\
\end{tabular}
\label{table 2}
\end{table}
\begin{figure}[!h]
    \centering
    \includegraphics[width=0.55\linewidth]{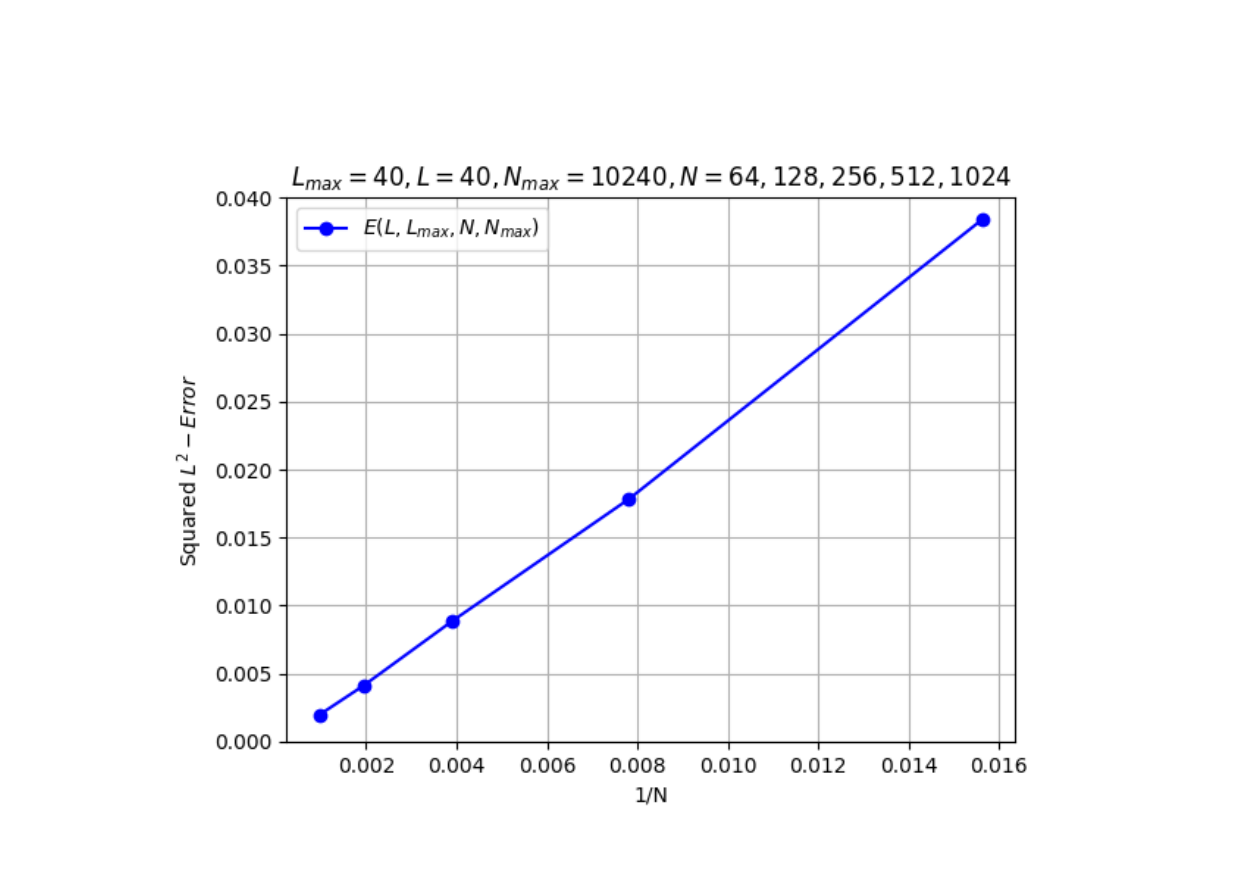}
    \hspace{-50pt}\includegraphics[width=0.55\linewidth]{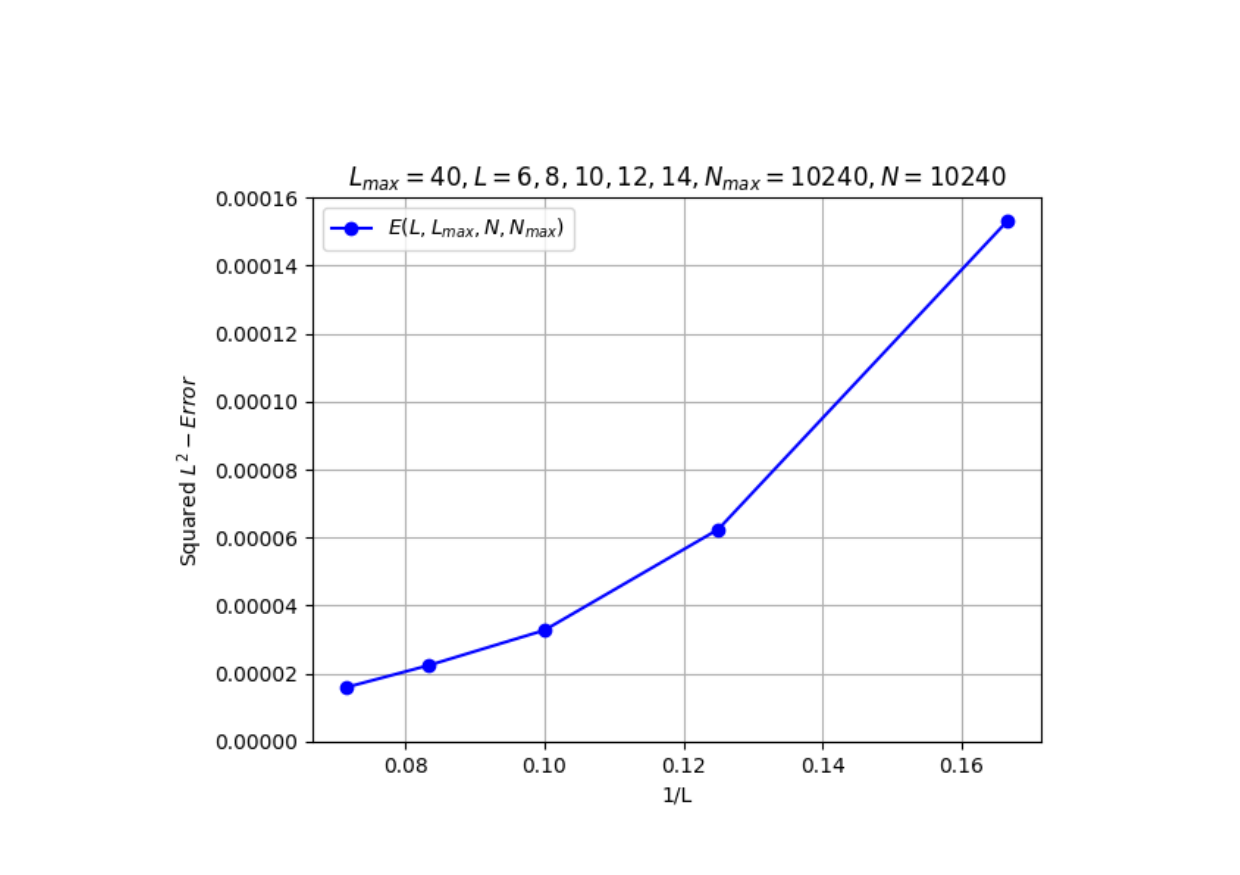}
    \caption{First graph: squared $L^2$-error with different time steps. Second graph: squared $L^2$-error with different spatial steps.}
    \label{figure 1}
\end{figure}
\begin{figure}[!h]
    \centering
    \includegraphics[width=0.55\linewidth]{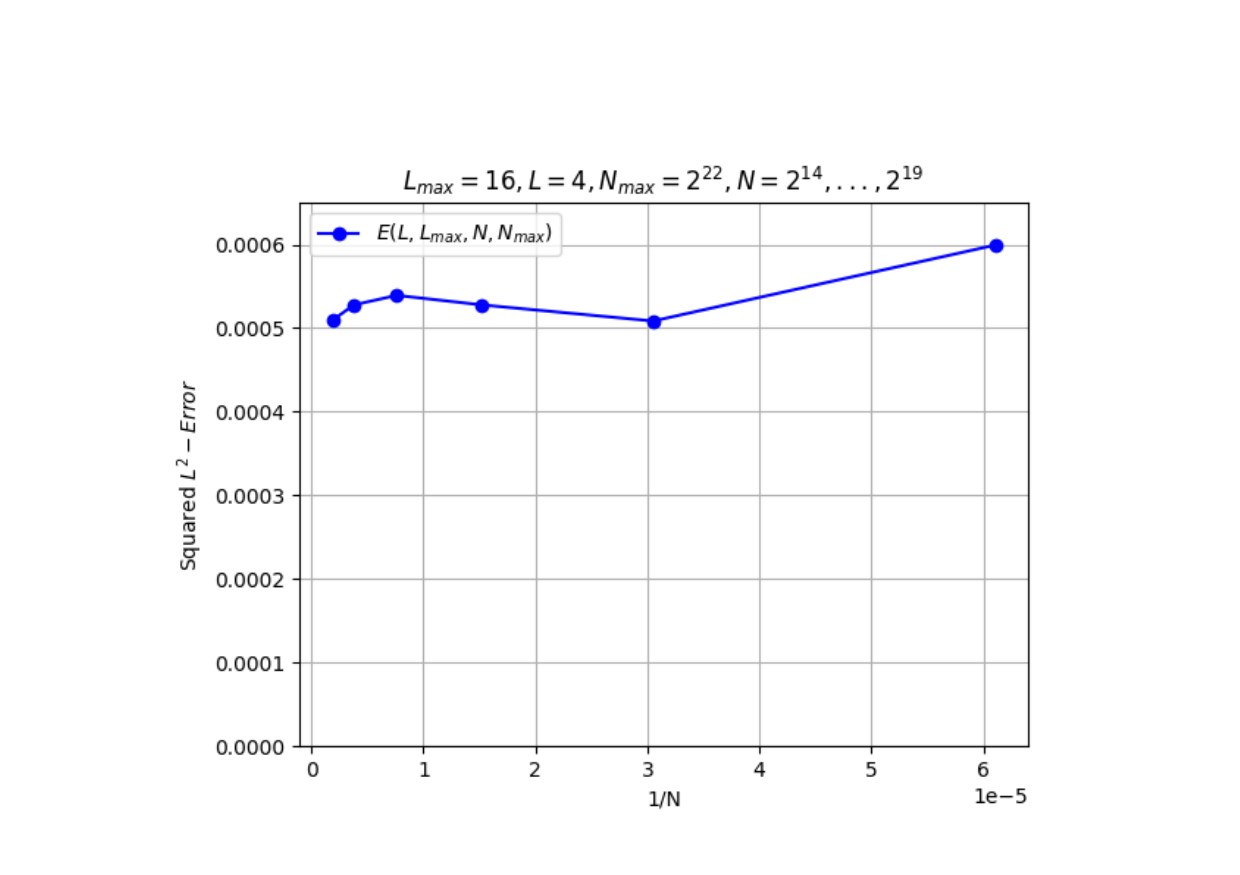}
    \hspace{-50pt}\includegraphics[width=0.55\linewidth]{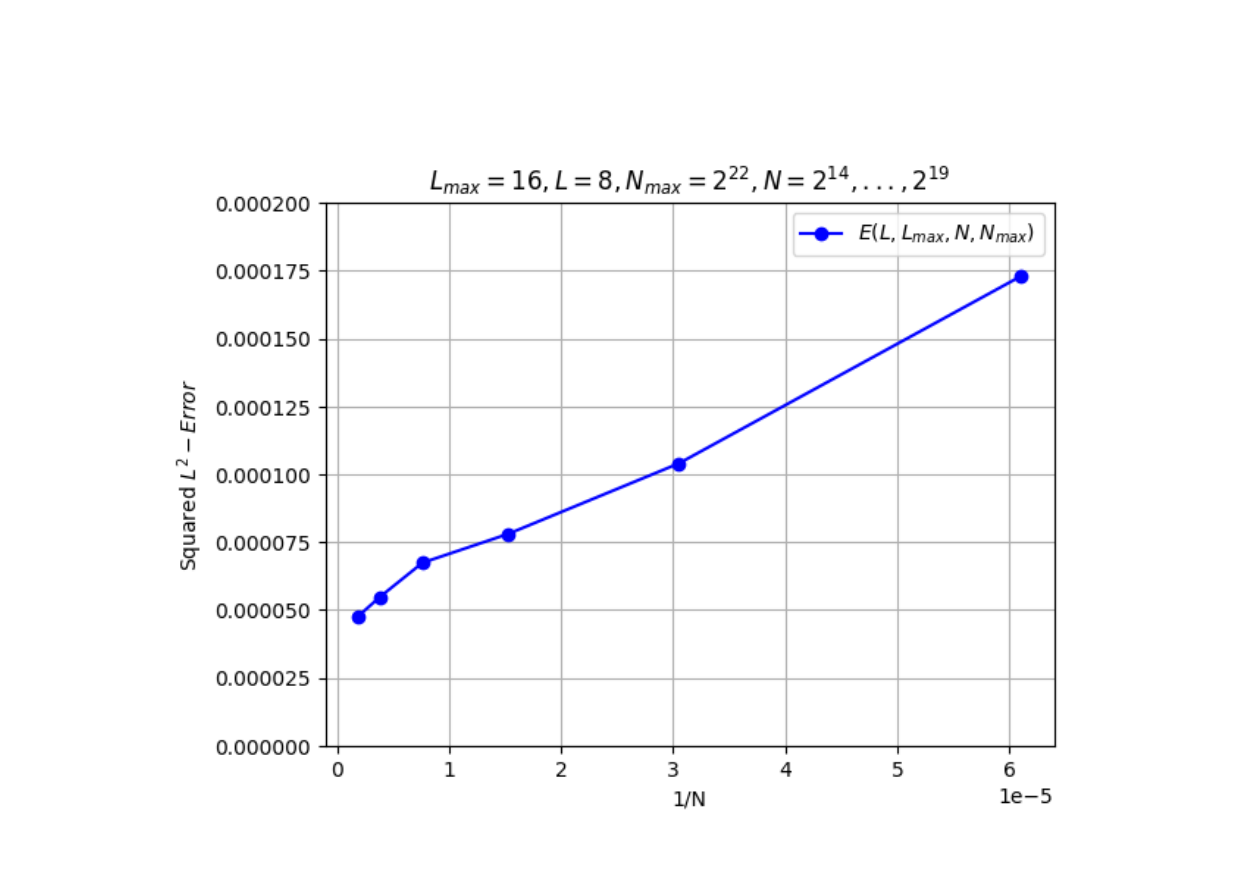}
    \caption{High number of time steps with $L_{max}=16$ and $L=4$ or $L=8$, respectively.}
    \label{figure 2}
\end{figure}

\section{Appendix}
\begin{lem}\label{Theorem 110124_01}
Let $d\in \{2,3\}$. Then, for any $u \in H^2(\Lambda)$ satisfying the weak homogeneous Neumann boundary condition we have
\begin{align*}
\Vert u \Vert_{H^2(\Lambda)}^2 \leq 12(\Vert \Delta u \Vert_2^2 + \Vert u \Vert_2^2),
\end{align*}
where $\Delta$ denotes the Laplace operator on $H^1(\Lambda)$ associated with the weak formulation of the homogeneous Neumann boundary condition.\\
Especially, for any random variable $u : \Omega \to H^2(\Lambda)$ we obtain
\begin{align*}
\Vert u \Vert_{H^2(\Lambda)}^2 \leq 12(\Vert \Delta u \Vert_2^2 + \Vert u \Vert_2^2) ~~~\mathbb{P}\text{-a.s. in}~\Omega.
\end{align*}
\end{lem}
\begin{proof}
Let $u \in H^2(\Lambda)$ satisfying the weak homogeneous Neumann boundary conditions. We set $f:= -\Delta u + u \in L^2(\Lambda)$. Then, according to Theorem 3.2.1.3 in \cite{Grisvard85} $u$ is the unique solution of $-\Delta u + u = f$ satisfying the weak homogeneous Neumann boundary conditions. Now, we follow the ideas of the proof of Theorem 3.2.1.3 in \cite{Grisvard85}. There exists a sequence $(\Lambda_m)_m$ such that $\Lambda_m\subset \mathbb{R}^d$ is bounded, open, convex, $\Lambda \subset \Lambda_m$, $\operatorname{dist}(\partial \Lambda, \partial \Lambda_m) \to 0$ as $m \to \infty$ and $\Lambda_m$ has a $\mathcal{C}^2$-boundary for any $m \in \mathbb{N}$. We set
\begin{align*}
\tilde{f}:= \begin{cases}
f, ~&\text{in}~ \Lambda, \\
0, ~&\text{in}~ \mathbb{R}^d \setminus \Lambda.
\end{cases}
\end{align*}
Then, since $\tilde{f} \in L^2(\Lambda_m)$ for any $ m \in \mathbb{N}$, there exists a unique $u_m \in H^2(\Lambda_m)$ satisfying the weak homogeneous Neumann boundary conditions w.r.t. $\Lambda_m$ such that $-\Delta u_m + u_m = \tilde{f}$ in $\Lambda_m$. According to the proof of Theorem 3.2.1.3 in \cite{Grisvard85} there exists $C>0$ such that $\Vert u_m \Vert_{H^2(\Lambda_m)} \leq C$ for any $m \in \mathbb{N}$ and $(u_m)_{|\Lambda} \rightharpoonup u$ in $H^2(\Lambda)$ for a subsequence. Now, Theorem 3.1.2.3 in \cite{Grisvard85} yields
\begin{align*}
\Vert u_m \Vert_{H^2(\Lambda_m)} \leq \sqrt{6} \Vert - \Delta u_m + u_m \Vert_{L^2(\Lambda_m)}.
\end{align*}
Hence, we obtain
\begin{align*}
\Vert u \Vert_{H^2(\Lambda)} &\leq \liminf\limits_{m \to \infty} \Vert u_m \Vert_{H^2(\Lambda)}\\
&\leq \liminf\limits_{m \to \infty} \Vert u_m \Vert_{H^2(\Lambda_m)}\\
&\leq \liminf\limits_{m \to \infty} \sqrt{6} \Vert - \Delta u_m + u_m \Vert_{L^2(\Lambda_m)}\\
&= \liminf\limits_{m \to \infty} \sqrt{6} \Vert \tilde{f} \Vert_{L^2(\Lambda_m)}\\
&= \sqrt{6} \Vert f \Vert_{L^2(\Lambda)}\\
&= \sqrt{6} \Vert - \Delta u + u \Vert_{L^2(\Lambda)} \\
&\leq \sqrt{6} (\Vert \Delta u \Vert_{L^2(\Lambda)} + \Vert u \Vert_{L^2(\Lambda)})
\end{align*}
and therefore we may conclude
\begin{align*}
\Vert u \Vert_{H^2(\Lambda)}^2 \leq 6 (\Vert \Delta u \Vert_{L^2(\Lambda)} + \Vert u \Vert_{L^2(\Lambda)})^2 \leq 12 (\Vert \Delta u \Vert_{L^2(\Lambda)}^2 + \Vert u \Vert_{L^2(\Lambda)}^2)
\end{align*}
\end{proof}
\begin{lem}\label{Lemma 240124_01}
Let $u \in H^2(\Lambda)$, $\mathcal{T}$ an admissible mesh, $K \in \mathcal{T}$ and $y \in K$. Then, for any $ 1 \leq q \leq 2$, there exists a constant $C=C(q, \Lambda)>0$ such that
\begin{align*}
\int_K |u(x) - u(y) |^q \, dx \leq C h^q \Vert u \Vert_{W^{2,q}(K)}^q.
\end{align*}
Especially, for any function $v$ of the form $v(x):= \sum\limits_{K \in \mathcal{T}} u(y_K) \mathds{1}_{K}(x)$, where $y_K \in K$ and $x \in \Lambda$, we have
\begin{align*}
\Vert u - v \Vert_{L^2(\Lambda)}^2 \leq C(2,\Lambda) h^2 \Vert u \Vert_{H^2(\Lambda)}^2.
\end{align*}
\end{lem}
\begin{proof}
Let $x,y \in K$. Then we have
\begin{align*}
u(x) - u(y) = \nabla u(x) \cdot (x-y) + \int_0^1 (1-s) (D^2(u)((1-s)y + sx)(x-y)) \cdot (x-y) \, ds. 
\end{align*}
Jensen inequality yields:
\begin{align*}
\int_K |u(x) - u(y) |^q \, dx &\leq 2^{q-1}  \int_K |\nabla u(x) (x-y) |^q \, dx \\&
+ 2^{q-1} \int_K \int_0^1 |(1-s) (D^2(u)((1-s)y + sx)(x-y)) \cdot (x-y)|^q \, ds \, dx.
\end{align*}
Now, a change of variables $z:= (1-s)y + sx$ yields
\begin{align*}
\int_K |u(x) - u(y) |^q \, dx &\leq 2^{q-1}  h^q \Vert \nabla u \Vert_{L^q(K)}^q + 2^{q-1} h^{2q}\Vert D^2(u) \Vert_{L^q(K)}^q \\
&\leq C h^q \Vert u \Vert_{W^{2,q}(K)}
\end{align*}
for $C=C(q, \Lambda)= 2^{q-1}(1+ \operatorname{diam}(\Lambda)^q)$.
\end{proof}

\section*{Acknowledgments}
The authors would like to thank Andreas Prohl for his valuable suggestions.

\bibliographystyle{plain}
\bibliography{Submission_Sapountzoglou_Zimmermann_CMAM_revision.bib}

\end{document}